\documentclass[11pt,reqno]{amsart}

\newtheorem{proposition}{Proposition}[section]
\newtheorem{lemma}[proposition]{Lemma}
\newtheorem{corollary}[proposition]{Corollary}
\newtheorem{theorem}[proposition]{Theorem}

\theoremstyle{definition}
\newtheorem{definition}[proposition]{Definition}
\newtheorem{example}[proposition]{Example}

\newtheorem{remark}[proposition]{Remark}

\newcommand{\thlabel}[1]{\label{th:#1}}
\newcommand{\thref}[1]{Theorem~\ref{th:#1}}
\newcommand{\selabel}[1]{\label{se:#1}}
\newcommand{\seref}[1]{Section~\ref{se:#1}}
\newcommand{\lelabel}[1]{\label{le:#1}}
\newcommand{\leref}[1]{Lemma~\ref{le:#1}}
\newcommand{\prlabel}[1]{\label{pr:#1}}
\newcommand{\prref}[1]{Proposition~\ref{pr:#1}}
\newcommand{\colabel}[1]{\label{co:#1}}
\newcommand{\coref}[1]{Corollary~\ref{co:#1}}
\newcommand{\relabel}[1]{\label{re:#1}}
\newcommand{\reref}[1]{Remark~\ref{re:#1}}
\newcommand{\exlabel}[1]{\label{ex:#1}}
\newcommand{\exref}[1]{Example~\ref{ex:#1}}
\newcommand{\delabel}[1]{\label{de:#1}}
\newcommand{\deref}[1]{Definition~\ref{de:#1}}
\newcommand{\eqlabel}[1]{\label{eq:#1}}
\newcommand{\equref}[1]{(\ref{eq:#1})}

\def\ot{\otimes}

\newcommand{\Cc}{\mathcal{C}}

\def\*C{{}^*\hspace*{-1pt}{\Cc}}
\def\text#1{{\rm {\rm #1}}}

\input xy
\xyoption {all} \CompileMatrices

\usepackage{amssymb}
\usepackage{color,amssymb,graphicx,amscd,amsmath}
\usepackage[colorlinks,urlcolor=blue,linkcolor=blue,citecolor=blue]{hyperref}

\begin{document}
\title[Unified products for Leibniz algebras]
{Unified products for Leibniz algebras. Applications}

\author{A. L. Agore}
\address{Faculty of Engineering, Vrije Universiteit Brussel, Pleinlaan 2, B-1050 Brussels, Belgium}
\email{ana.agore@vub.ac.be and ana.agore@gmail.com}

\author{G. Militaru}
\address{Faculty of Mathematics and Computer Science, University of Bucharest, Str.
Academiei 14, RO-010014 Bucharest 1, Romania}
\email{gigel.militaru@fmi.unibuc.ro and gigel.militaru@gmail.com}

\subjclass[2010]{16T10, 16T05, 16S40}

\thanks{A.L. Agore is research fellow ''Aspirant'' of FWO-Vlaanderen.
This work was supported by a grant of the Romanian National
Authority for Scientific Research, CNCS-UEFISCDI, grant no.
88/05.10.2011.}

\subjclass[2010]{17A32, 17B05, 17B56} \keywords{The extension and
the factorization problem, non-abelian cohomology for Leibniz
algebras, complements}


\begin{abstract}
Let $\mathfrak{g}$ be a Leibniz algebra and $E$ a vector space
containing $\mathfrak{g}$ as a subspace. All Leibniz algebra
structures on $E$ containing $\mathfrak{g}$ as a subalgebra are
explicitly described and classified by two non-abelian
cohomological type objects: ${\mathcal H}{\mathcal
L}^{2}_{\mathfrak{g}} \, (V, \, \mathfrak{g} )$ provides the
classification up to an isomorphism that stabilizes $\mathfrak{g}$
and ${\mathcal H}{\mathcal L}^{2} \, (V, \, \mathfrak{g} )$ will
classify all such structures from the view point of the extension
problem - here $V$ is a complement of $\mathfrak{g}$ in $E$. A
general product, called the unified product, is introduced as a
tool for our approach. The crossed (resp. bicrossed) products
between two Leibniz algebras are introduced as special cases of
the unified product: the first one is responsible for the
extension problem while the bicrossed product is responsible for
the factorization problem. The description and the classification
of all complements of a given extension $\mathfrak{g} \subseteq
\mathfrak{E} $ of Leibniz algebras are given as a converse of the
factorization problem. They are classified by another
cohomological object denoted by ${\mathcal H}{\mathcal A}^{2}
(\mathfrak{h}, \mathfrak{g} \, | \, (\triangleright,
\triangleleft, \leftharpoonup, \rightharpoonup))$, where
$(\triangleright, \triangleleft, \leftharpoonup ,
\rightharpoonup)$ is the canonical matched pair associated to a
given complement $\mathfrak{h}$. Several examples are worked out
in details.
\end{abstract}

\maketitle

\section*{Introduction}
Leibniz algebras were introduced by Bloh \cite{Bl1} under the name
of D-algebras and rediscovered later on by Loday \cite{Lod2} as
non-commutative generalizations of Lie algebras. A systematic
study of Leibniz algebras was initiated in \cite{LoP}, \cite{cu}.
Since then, Leibniz algebras generated a lot of interest and
became a field of study in its own right. Several classical
theorems known in the context of Lie algebras were extended to
Leibniz algebras, there exists a (co)homology theory for them, the
classification of certain types of Leibniz algebras of a given
(small) dimension was recently performed, their interaction with
vertex operator algebras, the Godbillon-Vey invariants for
foliations or differential geometry was highlighted. For more
details and motivations we refer to \cite{AOR}, \cite{AAO},
\cite{Bar}, \cite{ACM}, \cite{cam}, \cite{CK}, \cite{CC},
\cite{CI}, \cite{FM}, \cite{Go}, \cite{hu}, \cite{KW2000},
\cite{Lod2}, \cite{LoP}, \cite{mas}, \cite{RHa}, \cite{RHb} and
the references therein. The starting point of this paper is the
following question:

\textbf{Extending structures problem.} \textit{Let $\mathfrak{g}$
be a Leibniz  algebra and $E$ a vector space containing
$\mathfrak{g}$ as a subspace. Describe and classify the set of all
Leibniz algebra structures $[-, -]$ that can be defined on $E$
such that $\mathfrak{g}$ becomes a Leibniz subalgebra of $(E, [-,
-])$.}

In this paper, we will provide an answer to the above problem as
follows: first we will describe explicitly all Leibniz algebra
structures on $E$ which contain $\mathfrak{g}$ as a subalgebra;
then we will classify them up to a Leibniz algebra isomorphism
$\varphi : E \to E$ that \emph{stabilizes} $\mathfrak{g}$, that is
$\varphi$ acts as the identity on $\mathfrak{g}$. The extending
structures (ES) problem was formulated at the level of groups in
\cite{am-2010} and for arbitrary categories in \cite{am-2011}
where it was studied for quantum groups; recently we approached
the problem for Lie algebras \cite{am-2013}. The ES problem is a
very difficult one. If $\mathfrak{g} = \{0\}$, then the ES problem
asks for the classification of all Leibniz algebra structures on
an arbitrary vector space $E$, which is of course a hopeless
problem for vector spaces of large dimension: the classification
of all $3$-dimensional (resp. $4$-dimensional) Leibniz algebras
was finished only recently in \cite{CI} (resp. \cite{CK}). For
this reason, from now on we will assume that $\mathfrak{g} \neq
\{0\}$. Even though the ES problem is a difficult one, we can
still provide detailed answers to it in certain special cases
which depend on the choice of the Leibniz algebra $\mathfrak{g}$
and mainly on the codimension of $\mathfrak{g}$ in $E$. It
generalizes and unifies the \emph{extension problem} and the
\emph{factorization problem}. The extension problem asks for the
classification of all extensions of $\mathfrak{h}$ by
$\mathfrak{g}$ and it was first studied in \cite{LoP} for
$\mathfrak{g}$ abelian; in this case all such extensions are
classified by the second cohomology group ${\rm HL}^2
(\mathfrak{h}, \, \mathfrak{g})$ \cite[Proposition 1.9]{LoP}. The
fact that $\mathfrak{g}$ is abelian is essential in proving this
classification result. However, a classification result can still
be proved in the non-abelian case and the classification object
denoted by ${\mathbb H} {\mathbb L}^2 (\mathfrak{h}, \,
\mathfrak{g})$ will generalize the second cohomology group ${\rm
HL}^2 (\mathfrak{h}, \, \mathfrak{g})$ and it will be explicitly
constructed as a special case of the ES problem. The main drawback
of this construction is the fact that ${\mathbb H} {\mathbb L}^2
(\mathfrak{h}, \, \mathfrak{g})$ does not arise as a cohomology
group of a certain complex, it will be constructed using the
theory of crossed products for Leibniz algebras. To conclude, the
extension problem appears as a special case of the ES problem as
follows: if in the ES problem we replace the condition
''$\mathfrak{g}$ is a Leibniz subalgebra of $(E, [-, -])$'' by a
more restrictive one, namely ''$\mathfrak{g}$ is a two sided ideal
of $E$ and the quotient $E/ \mathfrak{g}$ is isomorphic to a given
Leibniz algebra $\mathfrak{h}$'', then what we obtain is in fact
the extension problem.

On the other hand, if in the ES problem we add the additional
hypothesis ''the complement of $\mathfrak{g}$ in $E$ is isomorphic
to a given Leibniz algebra $\mathfrak{h}$'' we obtain the
factorization problem for Leibniz algebras which can be explicitly
formulated as follows: describe and classify all Leibniz algebras
$\Xi$ that factorize through two given Leibniz algebras
$\mathfrak{g}$ and $\mathfrak{h}$, i.e. $\Xi$ contains
$\mathfrak{g}$ and $\mathfrak{h}$ as Leibniz subalgebras such that
$\Xi = \mathfrak{g} + \mathfrak{h}$ and $\mathfrak{g} \cap
\mathfrak{h} = \{0\}$. Exactly as in the case of Lie algebras
\cite[Theorem 4.1]{majid}, \cite[Theorem 3.9]{LW} we will
introduce the concept of a matched pair of Leibniz algebras and we
will associate to it a bicrossed product which will be responsible
for the factorization problem. However, in this case the
definition of the concept of a matched pair for Leibniz algebras
(\deref{mpLei}) is a lot more elaborated and difficult then the
one for Lie algebras. The last section is devoted to the converse
of the factorization problem. It consists of the following
question introduced in \cite{am-2013b} in the context of Hopf
algebras and Lie algebras:

\textbf{Classifying complements problem.} \textit{Let
$\mathfrak{g} \subseteq \mathfrak{E}$ be an extension of Leibniz
algebras. If a complement of $\mathfrak{g}$ in $\mathfrak{E}$
exists, describe explicitly and classify all complements of
$\mathfrak{g}$ in $\mathfrak{E}$, i.e. Leibniz subalgebras
$\mathfrak{h}$ of $\mathfrak{E}$ such that $\mathfrak{E} =
\mathfrak{g} + \mathfrak{h}$ and $\mathfrak{g} \cap \mathfrak{h} =
\{0\}$.}

The paper is organized as follows: in \seref{unifiedprod} we
introduce the abstract construction of the \emph{unified product}
$\mathfrak{g} \ltimes V$ for Leibniz algebras: it is associated to
a Leibniz algebra $\mathfrak{g}$, a vector space $V$ and a system
of data $\Omega(\mathfrak{g}, V) = \bigl(\triangleleft, \,
\triangleright, \, \leftharpoonup, \, \rightharpoonup,\, f, \{-,
\, -\} \bigl)$ called an extending datum of $\mathfrak{g}$ through
$V$. \thref{1} establishes the set of axioms that has to be
satisfied by $\Omega(\mathfrak{g}, V)$ such that $\mathfrak{g}
\ltimes V$ with a given bracket becomes a Leibniz algebra, i.e. is
a unified product. In this case, $\Omega(\mathfrak{g}, V) =
\bigl(\triangleleft, \, \triangleright, \, \leftharpoonup, \,
\rightharpoonup,\, f, \{-, \, -\} \bigl)$ will be called a
\emph{Leibniz extending structure} of $\mathfrak{g}$ through $V$.
Now let $\mathfrak{g}$ be a Leibniz algebra, $E$ a vector space
containing $\mathfrak{g}$ as a subspace and $V$ a given complement
of $\mathfrak{g}$ in $E$. \thref{classif} provides the answer to
the description part of the ES problem: there exists a Leibniz
algebra structure $[-,-]$ on $E$ such that $\mathfrak{g}$ is a
subalgebra of $(E, [-,-])$ if and only if there exists an
isomorphism of Leibniz algebras $(E, [-,-]) \cong \mathfrak{g}
\ltimes V$, for some Leibniz extending structure
$\Omega(\mathfrak{g}, V) = \bigl(\triangleleft, \, \triangleright,
\, \leftharpoonup, \, \rightharpoonup,\, f, \{-, \, -\} \bigl)$ of
$\mathfrak{g}$ through $V$. The theoretical answer to the
classification part of the ES problem is given in \thref{main1}:
we will construct explicitly a relative cohomology 'group',
denoted by ${\mathcal H} {\mathcal L}^{2}_{\mathfrak{g}} \, (V, \,
\mathfrak{g} )$, which will be the classifying object of all
extending structures of the Leibniz algebra $\mathfrak{g}$ to $E$
- the classification is given up to an isomorphism of Leibniz
algebras which stabilizes $\mathfrak{g}$. The construction of the
second classifying object, denoted by ${\mathcal H} {\mathcal
L}^{2} \, (V, \, \mathfrak{g} )$ is also given and it
parameterizes all extending structures of $\mathfrak{g}$ to $E$ up
to an isomorphism which simultaneously stabilizes $\mathfrak{g}$
and co-stabilizes $V$ - i.e. this classification is given from the
point of view of the extension problem. In \seref{exemple} we give
some explicit examples of computing ${\mathcal H}{\mathcal L}
^{2}_{\mathfrak{g}} \, (V, \, \mathfrak{g} )$ and ${\mathcal
H}{\mathcal L} ^{2} \, (V, \, \mathfrak{g} )$ in the case of
\emph{flag extending structures} as defined in \deref{flagex}. The
main result of this section is \thref{clasdim1}: several special
cases of it are discussed and explicit examples are given.
\seref{cazurispeciale} deals with two main special cases of the
unified product. The \emph{crossed product} of two Leibniz
algebras is introduced as a special case of the unified product.
\coref{croslieide} shows that any extension of a given Leibniz
algebra $\mathfrak{h}$ by a Leibniz algebra $\mathfrak{g}$ is
equivalent to a crossed product extension and the classifying
object ${\mathbb H} {\mathbb L}^2 (\mathfrak{h}, \, \mathfrak{g})$
of all extensions of $\mathfrak{h}$ by $\mathfrak{g}$ is
constructed in \coref{extprobra} as a generalization of the second
Loday-Pirashvili cohomology group ${\rm HL}^2 (\mathfrak{h}, \,
\mathfrak{g})$ \cite{LoP}. The concept of matched pair of Leibniz
algebras is introduced in \deref{mpLei} generalizing the one for
Lie algebras \cite[Theorem 3.9]{LW}, \cite[Theorem 4.1]{majid}. To
any matched pair of Leibniz algebras the \emph{bicrossed product}
is constructed as the tool responsible for the factorization
problem (\coref{bicrfactor}). Finally, \seref{complements} gives
the full answer to the classifying complements problem for Leibniz
algebras. The main result of the section is
\thref{clasformelorLie}: if $\mathfrak{g}$ is a Leibniz subalgebra
of $\mathfrak{E}$ and $\mathfrak{h}$ is a fixed complement of
$\mathfrak{g}$ in $\mathfrak{E}$, then the isomorphism classes of
all complements of $\mathfrak{g}$ in $\mathfrak{E}$ are
parameterized by a certain cohomological object denoted by
${\mathcal H} {\mathcal A}^{2} (\mathfrak{h}, \mathfrak{g} \, | \,
(\triangleright, \triangleleft, \leftharpoonup, \rightharpoonup)
)$ which is explicitly constructed, where $(\triangleright,
\triangleleft, \leftharpoonup, \rightharpoonup)$ is the
\emph{canonical matched pair} associated to the factorization
$\mathfrak{E} = \mathfrak{g} + \mathfrak{h}$. The key points in
proving this result are \thref{deforLie} and
\thref{descrierecomlie}.

\section{Preliminaries}\selabel{prel}
All vector spaces, linear or bilinear maps are over an arbitrary
field $k$.  A map $f: V \to W$ between two vector spaces is called
the trivial map if $f (v) = 0$, for all $v\in V$. Let
$\mathfrak{g} \leq E$ be a subspace in a vector space $E$; a
subspace $V$ of $E$ such that $E = \mathfrak{g} + V$ and $V \cap
\mathfrak{g} = 0$ is called a complement of $\mathfrak{g}$ in $E$.
Such a complement is unique up to an isomorphism and its dimension
is called the codimension of $\mathfrak{g}$ in $E$. A Leibniz
algebra is a vector space $\mathfrak{g}$, together with a bilinear
map $[- , \, -] : \mathfrak{g} \times \mathfrak{g} \to
\mathfrak{g}$ satisfying the Leibniz identity, that is:
\begin{equation}\eqlabel{lz}
[g, \, [h, \,l] \, ] = [ \, [g, \, h], \, l ] - [ \,[g, \, l], \,
h ]
\end{equation}
for all $g$, $h$, $l\in \mathfrak{g}$. Any Lie algebra is a
Leibniz algebra, and a Leibniz algebra satisfying $[g, \, g] = 0$,
for all $g \in \mathfrak{g}$ is a Lie algebra. The typical example
of a Leibniz algebra is the following \cite{Lod2}: let
$\mathfrak{g}$ be a Lie algebra, $(M, \triangleleft)$ a right
$\mathfrak{g}$-module and $\mu : M \to \mathfrak{g}$ a
$\mathfrak{g}$-equivariant map, i.e. $\mu (m \triangleleft g) =
[\mu (m), g]$, for all $m \in M$ and $g \in \mathfrak{g}$. Then
$M$ is a Leibniz algebra with the bracket $[m, n]_{(\triangleleft,
\, \mu)} := m \triangleleft \mu(n)$, for all $m$, $n\in M$.
Another important example was constructed in \cite{KP}: if
$\mathfrak{g}$ is a Lie algebra, then $\mathfrak{g} \ot
\mathfrak{g}$ is a Leibniz algebra with the bracket given by $[x
\ot y, a \ot b] := [x, [a, b]] \ot y + x \ot [y, [a, b]]$, for all
$x$, $y$, $a$, $b \in \mathfrak{g}$. For other interesting
examples of Leibniz algebras we refer to \cite{LoP}.

Let $\mathfrak{g}$ be a Leibniz algebra. A subspace $I \leq
\mathfrak{g}$ is called a two-sided ideal of $\mathfrak{g}$ if
$[x, g] \in I$ and $[g, x] \in I$, for all $x\in I$ and $g \in
\mathfrak{g}$. $\mathfrak{g}$ is called perfect if $[ \mathfrak{g}
, \mathfrak{g}] = \mathfrak{g}$ and abelian if $[ \mathfrak{g} ,
\mathfrak{g}] = 0$. By $Z( \mathfrak{g})$ we shall denote the
center of $\mathfrak{g}$, that is the two-sided ideal consisting
of all $g \in \mathfrak{g}$ such that $[g, \, x] = [x, \, g] = 0$,
for all $x \in \mathfrak{g}$. ${\rm Der} (\mathfrak{g})$ stands
for the space of all derivations of $\mathfrak{g}$, that is, all
linear maps $\Delta: \mathfrak{g} \to \mathfrak{g}$ such that for
any $g$, $h\in \mathfrak{g}$ we have
$$
\Delta ([g, \, h]) = [\Delta(g), \, h] + [g, \, \Delta(h)]
$$
A key role in the construction of the different non-abelian
cohomological objects arising from the classification part of the
ES problem will be played both by the classical space of
derivations and by the space of anti-derivations as defined below.

\begin{definition} \delabel{antiderivari}
An \emph{anti-derivation} of a Leibniz algebra $\mathfrak{g}$ is a
linear map $D : \mathfrak{g} \to \mathfrak{g}$ such that
\begin{equation}\eqlabel{antider}
D ([g, \, h]) = [D(g), \, h] - [D(h), \, g]
\end{equation}
for all $g$, $h\in \mathfrak{g}$. We denote by ${\rm ADer}
(\mathfrak{g})$ the space of all anti-derivations of
$\mathfrak{g}$.
\end{definition}

\begin{example} \exlabel{exemanti}
For a Lie algebra $\mathfrak{g}$ we have that ${\rm ADer}
(\mathfrak{g}) = {\rm Der} (\mathfrak{g})$ but in the case of
Leibniz algebras the two spaces are, in general, not equal. The
next example illustrates this. Let $\mathfrak{g}$ be the
$3$-dimensional Leibniz algebra with the basis $\{e_{1}, e_{2},
e_{3}\}$ and the bracket defined by: $ [e_{1}, \, e_{3}] = e_{2}$,
$[e_{3}, \, e_{3}] = e_{1}$. A straightforward computation shows
that the set ${\rm Der} (g)$ (resp. ${\rm ADer} (g)$) coincides
with the set of all arrays $\Delta$ (resp. $D$) of the form:
\begin{equation} \eqlabel{exnebu}
\Delta =
\left( \begin{array}{ccc} 2b_1 & 0 & b_2  \\
b_2 & 3b_1 & b_3 \\
0 & 0 & b_1  \\
\end{array}\right)
\qquad  ({\rm resp.} \quad  D =
\left( \begin{array}{ccc} 0 & 0 & d_1  \\
0 & 0 & d_2 \\
0 & 0 & d_3  \\
\end{array}\right)
\quad )
\end{equation}
for all $b_1$, $b_2$, $b_3$, $d_1$, $d_2$, $d_3 \in k$.
\end{example}

The following new concept will play an important role in the
construction of the two non-abelian cohomological objects
introduced in \seref{exemple} and \seref{cazurispeciale}.

\begin{definition} \delabel{dder}
Let $\mathfrak{g}$ be a Leibniz algebra. A \emph{pointed double
derivation} of $\mathfrak{g}$ is a triple $(g_{0}, \, D, \,
\Delta)$, where $g_{0} \in \mathfrak{g}$ and $D$, $\Delta:
\mathfrak{g} \to \mathfrak{g}$ are linear maps satisfying the
following compatibilities for any $g$, $h\in \mathfrak{g}$:
\begin{eqnarray}
&& D(g_0) = [g, \, g_0] = [g, D(h) + \Delta (h) ] =
D^2(g) + D \bigl( \Delta(g)\bigl) = 0 \eqlabel{flagper1} \\
&& D^2(g) + \Delta \bigl(D(g)\bigl) = [g_0, \, g] \eqlabel{flagper2} \\
&& \Delta \bigl([g, \, h]\bigl) = [\Delta(g), \, h] + [g, \, \Delta(h)] \eqlabel{flagper3} \\
&& D\bigl([g, \, h]\bigl) = [D(g), \, h] - [D(h), \, g]
\eqlabel{flagper4}
\end{eqnarray}
We denote by ${\mathcal D} (\mathfrak{g})$ the space of all
pointed double derivations.
\end{definition}
The compatibility conditions \equref{flagper3}-\equref{flagper4}
show that $\Delta$ (resp. $D$) is a derivation (resp. an
antiderivation) of $\mathfrak{g}$, hence ${\mathcal D} \,
(\mathfrak{g}) \subseteq \mathfrak{g} \times {\rm ADer}
(\mathfrak{g}) \times {\rm Der} (\mathfrak{g})$. If $\mathfrak{g}$
is a Lie algebra then we can easily see that the space ${\mathcal
D} (\mathfrak{g})$ coincides with the set of all pairs $(g_0, D)
\in Z(\mathfrak{g}) \times {\rm Der} (\mathfrak{g})$ such that
$D(g_0) = 0$. An example of computing the space ${\mathcal D}
(\mathfrak{g})$ is given in \exref{calexpext}.

Let $\mathfrak{g}$ be a Leibniz algebra, $E$ a vector space such
that $\mathfrak{g}$ is a subspace of $E$ and $V$ a complement of
$\mathfrak{g}$ in $E$, i.e. $V$ is a subspace of $E$ such that $E
= \mathfrak{g} + V$ and $V \cap \mathfrak{g} = 0$. For a linear
map $\varphi: E \to E$ we consider the diagram:
\begin{eqnarray} \eqlabel{diagrama}
\xymatrix {& \mathfrak{g} \ar[r]^{i} \ar[d]_{Id} & {E}
\ar[r]^{\pi} \ar[d]^{\varphi} & V \ar[d]^{Id}\\
& \mathfrak{g} \ar[r]^{i} & {E}\ar[r]^{\pi } & V}
\end{eqnarray}
where $\pi : E \to V$ is the canonical projection of $E =
\mathfrak{g} + V$ on $V$ and $i: \mathfrak{g} \to E$ is the
inclusion map. We say that $\varphi: E \to E$ \emph{stabilizes}
$\mathfrak{g}$ (resp. \emph{co-stabilizes} $V$) if the left square
(resp. the right square) of the diagram \equref{diagrama} is
commutative.

Two Leibniz algebra structures $\{-, \, -\}$ and $\{-, \, -\}'$ on
$E$ containing $\mathfrak{g}$ as a Leibniz subalgebra are called
\emph{equivalent} and we denote this by $(E, \{-, \, -\}) \equiv
(E, \{-, \, -\}')$, if there exists a Leibniz algebra isomorphism
$\varphi: (E, \{-, \, -\}) \to (E, \{-, \, -\}')$ which stabilizes
$\mathfrak{g}$. $\{-, \, -\}$ and $\{-, \, -\}'$ are called
\emph{cohomologous}, and we denote this by $(E, \{-, \, -\})
\approx (E, \{-, \, -\}')$, if there exists a Leibniz algebra
isomorphism $\varphi: (E, \{-, \, -\}) \to (E, \{-, \, -\}')$
which stabilizes $\mathfrak{g}$ and co-stabilizes $V$, i.e. the
diagram \equref{diagrama} is commutative.

$\equiv$ and $\approx$ are both equivalence relations on the set
of all Leibniz algebras structures on $E$ containing
$\mathfrak{g}$ as subalgebra and we denote by ${\rm ExtdL} \, (E,
\mathfrak{g})$ (resp. ${\rm ExtdL}' \, (E, \mathfrak{g})$) the set
of all equivalence classes via $\equiv$ (resp. $\approx$). ${\rm
ExtdL} \, (E, \mathfrak{g})$ is the classifying object of the ES
problem: by explicitly computing ${\rm ExtdL} \, (E,
\mathfrak{g})$ we obtain a parametrization of the set of all
isomorphism classes of Leibniz algebra structures on $E$ that
stabilize $\mathfrak{g}$. ${\rm ExtdL}' \, (E, \mathfrak{g})$
gives a classification of the ES problem from the point of view of
the extension problem. Any two cohomologous brackets on $E$ are of
course equivalent, hence there exists a canonical projection:
$$
{\rm ExtdL}' \,  (E, \mathfrak{g}) \twoheadrightarrow {\rm ExtdL}
\, (E, \mathfrak{g})
$$

For two sets $X$ and $Y$ we shall denote by $X \sqcup Y$ the
coproduct in the category of sets of $X$ and $Y$, i.e. $X \sqcup
Y$ is the disjoint union of $X$ and $Y$.

\section{Unified products for Leibniz algebras}\selabel{unifiedprod}
In this section we shall give the theoretical answer to the
ES-problem by constructing two cohomological type objects which
will parameterize ${\rm ExtdL} \, (E, \, \mathfrak{g})$ and ${\rm
ExtdL}' \, (E, \, \mathfrak{g})$. First we introduce the
following:

\begin{definition}\delabel{exdatum}
Let $\mathfrak{g}$ be a Leibniz algebra and $V$ a vector space. An
\textit{extending datum of $\mathfrak{g}$ through $V$} is a system
$\Omega(\mathfrak{g}, V) = \bigl(\triangleleft, \, \triangleright,
\, \leftharpoonup, \, \rightharpoonup, \, f, \, \{-,\,-\} \bigl)$
consisting of six bilinear maps
\begin{eqnarray*}
\triangleleft : V \times \mathfrak{g} \to V, \quad
&&\triangleright : V \times \mathfrak{g} \to \mathfrak{g}, \quad
\leftharpoonup \, : \mathfrak{g} \times V \to \mathfrak{g}, \quad
\rightharpoonup \, : \mathfrak{g} \times V \to V \\
&& f: V\times V \to \mathfrak{g}, \quad \{-,\, - \} \, : V\times V
\to V
\end{eqnarray*}
Let $\Omega(\mathfrak{g}, V) = \bigl(\triangleleft, \,
\triangleright, \, \leftharpoonup, \, \rightharpoonup, \, f, \,
\{-,\,-\} \bigl)$ be an extending datum. We denote by $
\mathfrak{g} \, \ltimes_{\Omega(\mathfrak{g}, V)} V = \mathfrak{g}
\, \ltimes V$ the vector space $\mathfrak{g} \, \times V$ together
with the bilinear map $[ -, \, -] : (\mathfrak{g} \times V) \times
(\mathfrak{g} \times V) \to \mathfrak{g} \times V$ defined by:
\begin{equation}\eqlabel{brackunif}
[(g, x), \, (h, y)] := \bigl( [g, \, h] + x \triangleright h + g
\leftharpoonup y + f(x, y), \,\, \{x, \, y \} + x\triangleleft h +
g \rightharpoonup y \bigl)
\end{equation}
for all $g$, $h \in \mathfrak{g}$ and $x$, $y \in V$. The object
$\mathfrak{g} \ltimes V$ is called the \textit{unified product} of
$\mathfrak{g}$ and $\Omega(\mathfrak{g}, V)$ if it is a Leibniz
algebra with the bracket given by \equref{brackunif}. In this case
the extending datum $\Omega(\mathfrak{g}, V) =
\bigl(\triangleleft, \, \triangleright, \, \leftharpoonup, \,
\rightharpoonup, \, f, \, \{-,\,-\} \bigl)$ is called a
\textit{Leibniz extending structure} of $\mathfrak{g}$ through
$V$. The maps $\triangleleft$, $\triangleright$, $\rightharpoonup$
and $\leftharpoonup$ are called the \textit{actions} of
$\Omega(\mathfrak{g}, V)$ and $f$ is called the \textit{cocycle}
of $\Omega(\mathfrak{g}, V)$.
\end{definition}

\begin{example}\exlabel{KW}
The unified product is a very general construction: in particular,
the \emph{hemisemidirect product} introduced in differential
geometry \cite[Example 2.2]{KW2000} is a special case of it. Let
$\mathfrak{g}$ be a Lie algebra and $(V, \triangleleft)$ be a
right $\mathfrak{g}$-module. Then $\mathfrak{g} \times V$ is a
Leibniz algebra with the bracket $[ (g, x), \, (h, y)] :=
\bigl([g, h], \, x \triangleleft h \bigl)$, for all $g$, $h\in
\mathfrak{g}$, $x$, $y\in V$ called the hemisemidirect product of
$\mathfrak{g}$ and $V$. This Leibniz algebra is not a Lie algebra
if $\mathfrak{g}$ acts nontrivially on $V$. The hemisemidirect
product is a special case of the unified product if we let
$\triangleright$, $\leftharpoonup$, $\rightharpoonup$, $f$ and
$\{-, - \}$ to be the trivial maps and $\mathfrak{g}$ to be a Lie
algebra.
\end{example}

Let $\Omega(\mathfrak{g}, V)$ be an extending datum of
$\mathfrak{g}$ through $V$. The bracket defined by
\equref{brackunif} has a rather complicated formula; however, for
some specific elements we obtain easier forms which will be very
useful for future computations. More precisely, the following
relations hold in $\mathfrak{g} \ltimes V$ for any $g$, $h \in
\mathfrak{g}$, $x$, $y \in V$:
\begin{eqnarray}
\left[(g, 0), \, (h, 0)\right] &=& \bigl(\left[g, \, h \right], \,
0 \bigl), \,\,\,\,\,\,\,\,\,\,\,\, \left[(g, 0), \, (0, y)\right]
= \bigl(g \leftharpoonup y, \, g \rightharpoonup y \bigl) \eqlabel{001}\\
\left[(0, x), \, (h, 0)\right] &=& \bigl( x \triangleright h, \, x
\triangleleft h \bigl), \,\,\,\ \left[(0, x), \, (0, y)\right] =
\bigl( f(x, y), \, \{x, y\} \bigl) \eqlabel{002}
\end{eqnarray}
The next theorem provides the set of axioms that need to be
fulfilled by an extending datum $\Omega(\mathfrak{g}, V)$ such
that $\mathfrak{g} \ltimes V$ is a unified product.

\begin{theorem}\thlabel{1}
Let $\mathfrak{g}$ be a Leibniz algebra, $V$ a vector space and
$\Omega(\mathfrak{g}, V)$ an extending datum of $\mathfrak{g}$ by
$V$. Then $\mathfrak{g} \ltimes V$ is a unified product if and
only if the following compatibility conditions hold for any $g$,
$h \in \mathfrak{g}$, $x$, $y$, $z \in V$:
\begin{enumerate}
\item[(L1)] $(V, \triangleleft)$ is a right $\mathfrak{g}$-module,
i.e. $x \lhd [g, \, h] = (x \lhd g) \lhd h - (x \lhd h) \lhd g$

\item[(L2)] $x \rhd [g, \, h] = [x \rhd g, \, h] - [x \rhd h, \,
g] + (x \lhd g) \rhd h - (x \lhd h) \rhd g$

\item[(L3)] $[g, \, h] \rightharpoonup x = g \rightharpoonup (h
\rightharpoonup x) + (g \rightharpoonup x) \lhd h$

\item[(L4)] $[g, \, h] \leftharpoonup x = [g, \, h \leftharpoonup
x] + [g \leftharpoonup x, \, h] + g \leftharpoonup (h
\rightharpoonup x) + (g \rightharpoonup x) \rhd h$

\item[(L5)] $x \rhd f(y, \, z) = f(x, \, y) \leftharpoonup z \, -
f(x, \, z) \leftharpoonup y + f(\{x, \, y\}, \, z) - f(\{x, \,
z\}, \, y) - f(x, \, \{y, \, z\})$

\item[(L6)] $x \lhd f(y, \, z) = f(x, \, y) \rightharpoonup z -
f(x, \, z) \rightharpoonup y + \{\{x, \, y\}, \, z\} - \{\{x, \,
z\}, \, y\} - \{x, \, \{y, \, z\}\}$

\item[(L7)] $\{x, \, y\} \rhd g = x \rhd (y \rhd g) + (x \rhd g)
\leftharpoonup y + f(x, \, y \lhd g) + f(x \lhd g, \, y) - [f(x,
\, y), \, g]$

\item[(L8)] $\{x, \, y\} \lhd g = x \lhd (y \rhd g) + (x \rhd g)
\rightharpoonup y + \{x, \, y \lhd g\} + \{x \lhd g, \, y\}$

\item[(L9)] $g \rightharpoonup \{x,\, y \} = (g \leftharpoonup x)
\rightharpoonup y \, - (g \leftharpoonup y) \rightharpoonup x +
\{g \rightharpoonup x, \, y\} - \{g \rightharpoonup y, \, x\}$

\item[(L10)] $g \leftharpoonup \{x, \, y\} = (g \leftharpoonup x)
\leftharpoonup y - (g \leftharpoonup y) \leftharpoonup x + f(g
\rightharpoonup x, \, y) - f(g \rightharpoonup y, \, x) - [g, \,
f(x, y)]$

\item[(L11)] $[g, \, h \leftharpoonup x] + [g, \, x \rhd h] + g
\leftharpoonup (h \rightharpoonup x) + g \leftharpoonup (x \lhd h)
= 0$

\item[(L12)] $x \rhd (y \rhd g) + x \rhd (g \leftharpoonup y) +
f(x, \, y \lhd g) + f(x, \, g \rightharpoonup y) = 0$

\item[(L13)] $x \lhd (y \rhd g) + x \lhd (g \leftharpoonup y) +
\{x, \, y \lhd g\} + \{x, \, g \rightharpoonup y\} = 0$

\item[(L14)] $g \rightharpoonup (h \rightharpoonup x) + g
\rightharpoonup (x \lhd h) = 0$
\end{enumerate}
\end{theorem}

\begin{proof} The proof relies on a detailed analysis of the Leibniz
identity for the bracket given by \equref{brackunif}, similar to
the one provided in the proof of \cite[Theorem 2.2]{am-2013}
corresponding to the Lie algebra case. There are, however, two
significant differences which require some extra care, namely: the
bracket on the Leibniz algebra $\mathfrak{g}$ is not
anti-symmetric and the Leibniz identity is not invariant under
circular permutations. As the computations are rather long but
straightforward we will only indicate the essential steps of the
proof, the details being left to the reader. To start with, we
note that $\mathfrak{g} \ltimes V$ is a Leibniz algebra if and
only if Leibniz's identity holds, i.e.:
\begin{equation}\eqlabel{005}
\bigl[(g, x), \, [(h, y), \, (l, z)]\bigl] = \bigl[ \, [(g, x), \,
(h, y)], \, (l, z) \bigl] - \bigl[ \, [(g, x), \, (l, z)], \, (h,
y) \bigl]
\end{equation}
for all $g$, $h$, $l \in \mathfrak{g}$ and $x$, $y$, $z \in V$.
Since in $\mathfrak{g} \ltimes V$ we have $(g, x) = (g, 0) + (0,
x)$ it follows that \equref{005} holds if and only if it holds for
all generators of $\mathfrak{g} \ltimes V$, i.e. the set $\{(g, \,
0) ~|~ g \in \mathfrak{g}\} \cup \{(0, \, x) ~|~ x \in V\}$. Hence
we are left to deal with eight cases which are necessary and
sufficient for testing the compatibility condition \equref{005}.
First, we should notice that \equref{005} holds for the triple
$(g, 0)$, $(h, 0)$, $(l, 0)$, since in $\mathfrak{g} \ltimes V$ we
have that $[(g, 0), \, (h, 0)] = ([g,h], 0)$. Now, taking into
account \equref{001}, we obtain that \equref{005} holds for $(g,
0)$, $(h, 0)$, $(0, x)$ if and only if
\begin{eqnarray*}
\bigl( [g, \, h] \leftharpoonup x - [g \leftharpoonup x, \, h] -
(g \rightharpoonup x) \rhd h, \, [g, \, h] \rightharpoonup x - (g
\rightharpoonup x) \lhd h \bigl) = \\ \bigl( [g, \, h
\leftharpoonup x] + g \leftharpoonup (h \rightharpoonup x), \, g
\rightharpoonup (h \rightharpoonup x)\bigl)
\end{eqnarray*}
i.e. if and only if $(L3)$ and $(L4)$ hold. A similar computation
proves that the compatibility condition \equref{005} is fulfilled
for the triple $(g, 0)$, $(0, x)$, $(h, 0)$ if and only if:
\begin{eqnarray*}
\left[g, \, h \right] \leftharpoonup x &=& [g \leftharpoonup x, \,
h] - [g,\, x \rhd h] + (g \rightharpoonup x) \rhd h - g
\leftharpoonup (x
\lhd h)\\
\left[g, \, h \right] \rightharpoonup x &=& (g \rightharpoonup x)
\lhd h - g \rightharpoonup (x \lhd h)
\end{eqnarray*}
Taking into account that we are looking for a minimal and
independent set of axioms we obtain, assuming that $(L3)$ and
$(L4)$ hold, that \equref{005} is fulfilled for the triple $(g,
0)$, $(0, x)$, $(h, 0)$ if and only if $(L11)$ and $(L14)$ hold.
Next, it is straightforward to see that \equref{005} holds for
$(g, 0)$, $(0, x)$, $(0, y)$ if and only if $(L9)$ and $(L10)$
hold. In a similar manner, one can show that \equref{005} holds
for $(0, x)$, $(0, y)$, $(0, z)$ if and only if axiom $(L5)$ and
$(L6)$ are fulfilled. We are left with three more cases to study.
First, observe that \equref{005} holds for $(0, x)$, $(g, 0)$,
$(h, 0)$ if and only if
$$
\bigl( x \rhd [g, \,h],\, x \lhd [g, \, h] \bigl) = \bigl( [x \rhd
g, \, h] - [x \rhd h, \, g] + (x \lhd g) \rhd h - (x \lhd h) \rhd
g, \, (x \lhd g) \lhd h - (x \lhd h) \lhd g \bigl)
$$
which is equivalent to the fact that axioms $(L1)$ and $(L2)$
hold. Analogously, we can show that \equref{005} holds for $(0,
x)$, $(0, y)$, $(g, 0)$ if and only if $(L7)$ and $(L8)$ hold.
Finally, it is straightforward to see that \equref{005} holds for
$(0, x)$, $(g, 0)$, $(0, y)$ if and only if the following two
compatibilities are fulfilled:
\begin{eqnarray*}
\{x, \, y\} \rhd g &=& (x \rhd g) \leftharpoonup y + f(x \lhd g,
\, y) - [f(x, \, y), \, g] - x \rhd (g \leftharpoonup y) - f(x, \,
g \rightharpoonup y)\\
\{x, \, y\} \lhd g &=& \{x \lhd g, \, y\} + (x \rhd g)
\rightharpoonup y - x \lhd (g \leftharpoonup y) - \{x, \, g
\rightharpoonup y\}
\end{eqnarray*}
Since we are looking for a minimal set of axioms we obtain,
assuming that $(L7)$ and $(L8)$ hold, that \equref{005} is
fulfilled for the triple $(0, x)$, $(g, 0)$, $(0, y)$ if and only
if $(L12)$ and $(L13)$ hold. The proof is now finished.
\end{proof}

From now on, a Leibniz extending structure of $\mathfrak{g}$
through $V$ will be viewed as a system $\Omega(\mathfrak{g}, V) =
\bigl(\triangleleft, \, \triangleright, \, \leftharpoonup, \,
\rightharpoonup, \, f, \, \{-,\, -\} \bigl)$ satisfying the
compatibility conditions $(L1)-(L14)$. We denote by ${\mathcal L}
{\mathcal Z} (\mathfrak{g}, V)$ the set of all Leibniz extending
structures of $\mathfrak{g}$ through $V$.

\begin{example}\exlabel{twistedproduct}
We provide the first example of a Leibniz extending structure and
the corresponding unified product. More examples will be given in
\seref{exemple} and \seref{cazurispeciale}.

Let $\Omega(\mathfrak{g}, V) = \bigl(\triangleleft, \,
\triangleright, \, \leftharpoonup, \, \rightharpoonup, \, f, \{-,
\, -\} \bigl)$ be an extending datum of a Leibniz algebra
$\mathfrak{g}$ through a vector space $V$ such that its actions
are all the trivial maps, i.e. $x\triangleleft g = x
\triangleright g = g \rightharpoonup x = g \leftharpoonup x = 0$,
for all $x \in V$ and $g \in \mathfrak{g}$. Then,
$\Omega(\mathfrak{g}, V) = \bigl(f, \{-, \, -\} \bigl)$ is a
Leibniz extending structure of $\mathfrak{g}$ through $V$ if and
only if $(V, \{-, -\})$ is a Leibniz algebra and $f: V \times V
\to \mathfrak{g}$ is an abelian $2$-cocycle, that is
\begin{equation}\eqlabel{2cociclucltw}
[g, \, f(x, y)] = [f(x, y), \, g] = 0, \qquad f\bigl(x, \, \{y,\,
z \}\bigl) - f\bigl(\{x,\, y\}, \, z \bigl) + f\bigl(\{x, \, z \},
\,  y \bigl) = 0
\end{equation}
for all $g \in \mathfrak{g}$, $x$, $y$ and $z \in V$. The first
part of \equref{2cociclucltw} shows that the image of $f$ is
contained in the center of $\mathfrak{g}$, while the second part
is a $2$-cocycle condition (see \cite[Section 1.7]{LoP}) which
follows from the axiom $(L5)$. In this case, the associated
unified product $\mathfrak{g} \ltimes_{\Omega(\mathfrak{g}, V)} V$
will be called the \emph{twisted product} of the Leibniz algebras
$\mathfrak{g}$ and $V$.
\end{example}

Let $\Omega(\mathfrak{g}, V) = \bigl(\triangleleft, \,
\triangleright, \, \leftharpoonup, \, \rightharpoonup, \, f, \,
\cdot \bigl) \, \in {\mathcal L} {\mathcal Z} (\mathfrak{g}, V)$
be a Leibniz extending structure and $\mathfrak{g} \ltimes V$ the
associated unified product. Then the canonical inclusion
$$
i_{\mathfrak{g}}: \mathfrak{g} \to \mathfrak{g} \ltimes V, \qquad
i_{\mathfrak{g}}(g) = (g, \, 0)
$$
is an injective Leibniz algebra map. Therefore, we can see
$\mathfrak{g}$ as a subalgebra of $\mathfrak{g} \ltimes V$ through
the identification $\mathfrak{g} \cong
i_{\mathfrak{g}}(\mathfrak{g}) \cong \mathfrak{g} \times \{0\}$.
Conversely, we have the following result which provides an answer
to the description part of the ES problem:

\begin{theorem}\thlabel{classif}
Let $\mathfrak{g}$ be a Leibniz algebra, $E$ a vector space
containing $\mathfrak{g}$ as a subspace and $[-, \,-]$ a Leibniz
algebra structure on $E$ such that $\mathfrak{g}$ is a Leibniz
subalgebra in $(E, [-, \,-])$. Then there exists a Leibniz
extending structure $\Omega(\mathfrak{g}, V) =
\bigl(\triangleleft, \, \triangleright, \, \leftharpoonup, \,
\rightharpoonup, \, f, \,  [-, -] \bigl)$ of $\mathfrak{g}$
through a subspace $V$ of $E$ and an isomorphism of Leibniz
algebras $(E, [-, \,-]) \cong \mathfrak{g} \ltimes V$ that
stabilizes $\mathfrak{g}$ and co-stabilizes $V$.
\end{theorem}

\begin{proof}
Since $k$ is a field, there exists a linear map $p: E \to
\mathfrak{g}$ such that $p(g) = g$, for all $g \in \mathfrak{g}$.
Then $V := \rm{Ker}$$(p)$ is a complement of $\mathfrak{g}$ in
$E$. We define the extending datum $\Omega(\mathfrak{g}, V) =
\bigl(\triangleleft = \triangleleft_p , \, \triangleright =
\triangleright_p , \, \leftharpoonup = \leftharpoonup_p , \,
\rightharpoonup = \rightharpoonup_p, \, f = f_p, \, [-, -] = [-,
-]_p \bigl)$ of $\mathfrak{g}$ through $V$ by the following
formulas:
\begin{eqnarray*}
&& \triangleright : V \times \mathfrak{g} \to \mathfrak{g}, \quad
\,\,\, x \triangleright g = p \bigl([x, \,g]\bigl), \qquad
\triangleleft: V \times \mathfrak{g} \to V,
\quad \,\, x \triangleleft g = [x, \, g] - p \bigl([x, \, g]\bigl)\\
&& \leftharpoonup \, : \mathfrak{g} \times V \to \mathfrak{g},
\quad g \leftharpoonup x = p\bigl([g, \, x]\bigl), \quad
\rightharpoonup \, : \mathfrak{g} \times V \to V, \quad g
\rightharpoonup x = [g, \,
x] - p\bigl([g, \, x]\bigl)\\
&& f: V \times V \to \mathfrak{g}, \,\,\, f(x, y) = p \bigl([x, \,
y]\bigl), \quad  \{\, , \, \}: V \times V \to V, \,\, \{x, y\} =
[x, \, y] - p \bigl([x, \, y]\bigl)
\end{eqnarray*}
for all $g \in \mathfrak{g}$, $x$, $y\in V$. Now, the map
$\varphi: \mathfrak{g} \times V \to E$, $\varphi(g, x) := g+x$, is
a linear isomorphism between the direct product of vector spaces
$\mathfrak{g} \times V$ and the Leibniz algebra $(E, [-, -])$ with
the inverse given by $\varphi^{-1}(y) := \bigl(p(y), \, y -
p(y)\bigl)$, for all $y \in E$. Hence, there exists a unique
Leibniz algebra structure on $\mathfrak{g} \times V$ such that
$\varphi$ is an isomorphism of Leibniz algebras and this unique
bracket is given by:
$$
[(g, x), \, (h, y)] := \varphi^{-1} \bigl([\varphi(g, x), \,
\varphi(h, y)]\bigl)
$$
for all $g$, $h \in \mathfrak{g}$ and $x$, $y\in V$. Using
\thref{1}, the proof will be finished if we prove that this
bracket is the one defined by \equref{brackunif} associated to the
system $\bigl(\triangleleft_p, \, \triangleright_p, \,
\leftharpoonup_{p}, \, \rightharpoonup_{p}, \, f_p, \{-, \, -\}_p
\bigl)$ constructed above. Indeed, for any $g$, $h \in
\mathfrak{g}$, $x$, $y\in V$ we have:
\begin{eqnarray*}
[(g, x), \, (h, y)] &=& \varphi^{-1} \bigl([\varphi(g, x), \,
\varphi(h, y)]\bigl) \, = \varphi^{-1} \bigl( [ g+ x, \, h + y ]
\bigl) \\
&=& \varphi^{-1} \bigl( [g, h] + [g, y] + [x, h] + [x, y] \bigl) \\
&=& \Bigl( [g, h] + p ( [g, y]) + p( [x, h] ) + p( [x, y]), \\
&& [g, y] + [x, h] + [x, y] - p([g, y]) - p([x, h]) - p([x, y]) \Bigl) \\
&=& \bigl( [g, \, h] + x \triangleright h + g \leftharpoonup y +
f(x, y), \,\, \{x, \, y \} + x\triangleleft h + g \rightharpoonup
y \bigl)
\end{eqnarray*}
as needed. Thus, $\varphi: \mathfrak{g} \ltimes V \to E$ is an
isomorphism of Leibniz algebras and the following diagram is
commutative
\begin{eqnarray*}
\xymatrix {& \mathfrak{g} \ar[r]^{i} \ar[d]_{Id} & {\mathfrak{g}
\ltimes V} \ar[r]^{q} \ar[d]^{\varphi} & V \ar[d]^{Id}\\
& \mathfrak{g} \ar[r]^{i} & {E}\ar[r]^{\pi } & V}
\end{eqnarray*}
where $\pi : E \to V$ is the projection of $E = A + V$ on the
vector space $V$ and $q: A \ltimes V \to V$, $q (g, x) := x$ is
the canonical projection.
\end{proof}

Based on \thref{classif}, the classification of all Leibniz
algebra structures on $E$ that contain $\mathfrak{g}$ as a Leibniz
subalgebra reduces to the classification of all unified products
$\mathfrak{g} \ltimes V$, associated to all Leibniz extending
structures $\Omega(\mathfrak{g}, V) = \bigl(\triangleleft, \,
\triangleright, \leftharpoonup, \, \rightharpoonup, \, f, \{-, \,
-\} \bigl)$, for a given complement $V$ of $\mathfrak{g}$ in $E$.
In order to construct the cohomological objects ${\mathcal
H}{\mathcal L}^{2}_{\mathfrak{g}} \, (V, \, \mathfrak{g} )$ and
${\mathcal H}{\mathcal L}^{2} \, (V, \, \mathfrak{g} )$ which will
parameterize the classifying sets ${\rm ExtdL} \, (E,
\mathfrak{g})$ and respectively ${\rm ExtdL}' \, (E,
\mathfrak{g})$ we need the following technical result:

\begin{lemma} \lelabel{morfismunif}
Let $\Omega(\mathfrak{g}, V) = \bigl(\triangleleft, \,
\triangleright, \leftharpoonup, \, \rightharpoonup, \, f, \{-, \,
-\} \bigl)$ and $\Omega'(\mathfrak{g}, V) = \bigl(\triangleleft ',
\, \triangleright ', \leftharpoonup ', \, \rightharpoonup ', \,
f', \{-, \, -\}' \bigl)$ be two Leibniz extending structures of
$\mathfrak{g}$ trough $V$ and $ \mathfrak{g} \ltimes V$, $
\mathfrak{g} \ltimes ' V$ the associated unified products. Then
there exists a bijection between the set of all morphisms of
Leibniz algebras $\psi: \mathfrak{g} \ltimes V \to \mathfrak{g}
\ltimes ' V$ which stabilizes $\mathfrak{g}$ and the set of pairs
$(r, v)$, where $r: V \to \mathfrak{g}$, $v: V \to V$ are two
linear maps satisfying the following compatibility conditions for
any $g \in \mathfrak{g}$, $x$, $y \in V$:
\begin{enumerate}
\item[(ML1)] $v(g \rightharpoonup x) = g \rightharpoonup ' v(x)$;
\item[(ML2)] $v(x \triangleleft g) = v(x) \triangleleft ' g$;
\item[(ML3)] $x \triangleright g + r(x \triangleleft g) = [r(x),
\, g] + v(x) \triangleright ' g$; \item[(ML4)] $g \leftharpoonup x
+ r(g \rightharpoonup x) = [g, \, r(x)] + g \leftharpoonup '
v(x)$; \item[(ML5)] $v(\{x, \, y\}) = r(x) \rightharpoonup ' v(y)
+ v(x) \triangleleft ' r(y) + \{v(x), \, v(y)\}'$; \item[(ML6)]
$f(x, \, y) + r(\{x, \, y\}) = [r(x), \, r(y)] + r(x)
\leftharpoonup ' v(y) + v(x) \triangleright ' r(y) + f' (v(x), \,
v(y))$
\end{enumerate}
Under the above bijection the morphism of Leibniz algebras $\psi =
\psi_{(r, v)}: \mathfrak{g} \ltimes V \to \mathfrak{g} \ltimes '
V$ corresponding to $(r, v)$ is given for any $g \in \mathfrak{g}$
and $x \in V$ by:
$$
\psi(g, x) = (g + r(x), v(x))
$$
Moreover, $\psi = \psi_{(r, v)}$ is an isomorphism if and only if
$v: V \to V$ is an isomorphism and $\psi = \psi_{(r, v)}$
co-stabilizes $V$ if and only if $v = {\rm Id}_V$.
\end{lemma}

\begin{proof}
A linear map $\psi: \mathfrak{g} \ltimes V \to \mathfrak{g}
\ltimes ' V$ which stabilizes $\mathfrak{g}$ is uniquely
determined by two linear maps $r: V \to \mathfrak{g}$, $v: V \to
V$ such that $\psi(g, x) = (g + r(x), v(x))$, for all $g \in
\mathfrak{g}$, and $x \in V$. Indeed, by denoting $\psi(0, x) =
(r(x), v(x)) \in \mathfrak{g} \times V$ for all $x \in V$, we
obtain:
\begin{eqnarray*}
\psi(g, x) &=& \psi \bigl((g, 0) + \psi(0, x)\bigl) = \psi(g, 0) +
\psi(0, x) = \bigl(g + r(x), v(x) \bigl)
\end{eqnarray*}
Let $\psi = \psi_{(r, v)}$ be such a linear map, i.e. $\psi(g, x)
= (g + r(x), v(x))$, for some linear maps $r: V \to \mathfrak{g}$,
$v: V \to V$. We will prove that $\psi$ is a morphism of Leibniz
algebras if and only if the compatibility conditions $(ML1)-(ML6)$
hold. It is enough to prove that the compatibility
\begin{equation}\eqlabel{Liemap}
\psi \bigl([(g, x), \, (h, y)] \bigl) = [\psi(g, x), \, \psi(h,
y)]
\end{equation}
holds for all generators of $\mathfrak{g} \ltimes V$. First of
all, it is easy to see that \equref{Liemap} holds for the pair
$(g, 0)$, $(h, 0)$, for all $g$, $h \in \mathfrak{g}$. Now we
prove that \equref{Liemap} holds for the pair $(g, 0)$, $(0, x)$
if and only if $(ML1)$ and $(ML4)$ hold. Indeed, $\psi \bigl([(g,
0), \, (0, x)] \bigl) = [\psi(g, 0), \, \psi(0, x)]$ it is
equivalent to $\psi(g \leftharpoonup x, \, g \rightharpoonup x) =
[(g, 0), \, (r(x), v(x))]$ and hence to $(g \leftharpoonup x + r(g
\rightharpoonup x), \, v(g \rightharpoonup x)) = ([g, \, r(x)]+ g
\leftharpoonup ' v(x), \, g \rightharpoonup ' v(x))$, i.e. to the
fact that $(ML1)$ and $(ML4)$ hold.

Next we prove that \equref{Liemap} holds for the pair $(0, x)$,
$(g, 0)$ if and only if $(ML2)$ and $(ML3)$ hold. Indeed, $\psi
\bigl([(0, x), \, (g, 0)] \bigl) = [\psi(0, x), \, \psi(g, 0)]$ it
is equivalent to $\psi(x \triangleright g, \, x \triangleleft g) =
[(r(x), v(x)), \, (g, 0)]$ and therefore to $(x \triangleright g +
r(x \triangleleft g), \, v(x \triangleleft g)) = ([r(x),\, g] +
v(x) \triangleright  ' g, \, v(x) \triangleleft ' g)$, i.e. to the
fact that $(ML2)$ and $(ML3)$ hold.

To this end, we prove that \equref{Liemap} holds for the pair $(0,
x)$, $(0, y)$ if and only if $(ML5)$ and $(ML6)$ hold. Indeed,
$\psi \bigl([(0, x), \, (0, y)] \bigl) = [\psi(0, x), \, \psi(0,
y)]$ it is equivalent to $\psi(f(x, y),\\ \{x, y\}) = [(r(x),
v(x)), \, (r(y), v(y))]$; thus it is equivalent to: $\bigl(f(x, y)
+ r(\{x, y\}), v(\{x, y\})\bigl)\\ = \bigl([r(x), \, r(y)] + r(x)
\leftharpoonup ' v(y) + v(x) \triangleright ' r(y) + f'(v(x),
v(y)),\, r(x) \rightharpoonup ' v(y) + v(x) \triangleleft ' r(y) +
 \{v(x), v(y)\}'\bigl)$, i.e. to the fact that $(ML5)$ and $(ML6)$ hold.

Assume now that $v: V \to V$ is bijective. Then $\psi_{(r, v)}$ is
an isomorphism of Leibniz algebras with the inverse given by $
\psi_{(r, v)}^{-1}(h, y) = \bigl(h - r(v^{-1}(y)),
v^{-1}(y)\bigl)$, for all $h \in \mathfrak{g}$ and $y \in V$.
Conversely, assume that $\psi_{(r, v)}$ is bijective. It follows
easily that $v$ is surjective. Thus, we are left to prove that $v$
is injective. Indeed, let $x \in V$ such that $v(x) = 0$. We have
 $\psi_{(r, v)}(0, 0) = (0, 0) = (0, v(x)) = \psi_{(r, v)}(-
r(x), x)$, and hence we obtain $x = 0$, i.e. $v$ is a bijection.
The last assertion is trivial and the proof is now finished.
\end{proof}

\begin{definition}\delabel{echiaa}
Two Leibniz extending structures $\Omega(\mathfrak{g}, V) =
\bigl(\triangleleft, \, \triangleright, \leftharpoonup, \,
\rightharpoonup, \, f, \{-, \, -\} \bigl)$ and
$\Omega'(\mathfrak{g}, V) = \bigl(\triangleleft ', \,
\triangleright ', \leftharpoonup ', \, \rightharpoonup ', \, f',
\{-, \, -\}' \bigl)$ are called \emph{equivalent} and we denote
this by $\Omega(\mathfrak{g}, V) \equiv \Omega'(\mathfrak{g}, V)$,
if there exists a pair $(r, v)$ of linear maps, where $r: V \to
\mathfrak{g}$ and $v \in {\rm Aut}_{k}(V)$ such that
$\bigl(\triangleleft, \, \triangleright, \leftharpoonup, \,
\rightharpoonup, \, f, \{-, \, -\} \bigl)$ is implemented from
$\bigl(\triangleleft ', \, \triangleright ', \leftharpoonup ', \,
\rightharpoonup ', \, f', \{-, \, -\}' \bigl)$ using $(r, v)$ via:
\begin{eqnarray*}
x \triangleleft g &=& v^{-1} \bigl(v(x) \triangleleft ' g\bigl)\\
g \rightharpoonup x &=& v^{-1} \bigl(g \rightharpoonup ' v(x)\bigl)\\
x \triangleright g &=& [r(x), \, g] + v(x) \triangleright ' g - r \circ v^{-1}
\bigl(v(x) \triangleleft ' g\bigl) \\
g \leftharpoonup x &=& [g, \, r(x)] + g \leftharpoonup ' v(x) - r
\circ v^{-1} \bigl(g \rightharpoonup ' v(x)\bigl)\\
\{x, \, y\} &=& v^{-1} \bigl(r(x) \rightharpoonup ' v(y) + v(x)
\triangleleft ' r(y) + \{v(x), \, v(y)\}'\bigl)\\
f(x,\, y) &=& [r(x), \, r(y)] + r(x) \leftharpoonup ' v(y) + v(x)
\triangleright ' r(y) + f ' \bigl(v(x), \, v(y)\bigl) - \\
&& r \circ v^{-1} \bigl(r(x) \rightharpoonup ' v(y) + v(x)
\triangleleft ' r(y) + \{v(x), \, v(y)\}'\bigl)
\end{eqnarray*}
for all $g \in \mathfrak{g}$, $x$, $y \in V$.
\end{definition}

Using \leref{morfismunif}, we obtain that $\Omega(\mathfrak{g}, V)
\equiv \Omega'(\mathfrak{g}, V)$ if and only if there exists $\psi
: \mathfrak{g} \ltimes V \to \mathfrak{g} \ltimes' V$ an
isomorphism of Leibniz algebras that stabilizes $\mathfrak{g}$,
where $\mathfrak{g} \ltimes V$ and $\mathfrak{g} \ltimes' V$ are
the corresponding unified products. On the other hand, the
isomorphisms between two unified products that stabilize
$\mathfrak{g}$ and co-stabilize $V$ are decoded by the following:

\begin{definition}\delabel{echiaab}
Two Leibniz extending structures $\Omega(\mathfrak{g}, V) =
\bigl(\triangleleft, \, \triangleright, \leftharpoonup, \,
\rightharpoonup, \, f, \{-, \, -\} \bigl)$ and
$\Omega'(\mathfrak{g}, V) = \bigl(\triangleleft ', \,
\triangleright ', \leftharpoonup ', \, \rightharpoonup ', \, f',
\{-, \, -\}' \bigl)$ are called \emph{cohomologous} and we denote
this by $\Omega(\mathfrak{g}, V) \approx \Omega'(\mathfrak{g}, V)$
if and only if $\triangleleft \, = \, \triangleleft '$,
$\rightharpoonup \, = \,  \rightharpoonup '$ and there exists a
linear map $r: V \to \mathfrak{g}$ such that for any $g \in
\mathfrak{g}$, $x$, $y \in V$:
\begin{eqnarray*}
x \triangleright g &=& [r(x), \, g] + x \triangleright ' g - r(x \triangleleft ' g) \\
g \leftharpoonup x &=& [g, \, r(x)] + g \leftharpoonup ' x - r(g \rightharpoonup ' x)\\
\{x, \, y\} &=& r(x) \rightharpoonup ' y + x
\triangleleft ' r(y) + \{x, \, y\}'\\
f(x,\, y) &=& [r(x), \, r(y)] + r(x) \leftharpoonup ' y + x
\triangleright ' r(y) + f ' (x, \, y) -\\ &&r \bigl(r(x)
\rightharpoonup ' y + x \triangleleft ' r(y) + \{x, \, y\}'\bigl)
\end{eqnarray*}
\end{definition}

As a conclusion of this section, the theoretical answer to the
ES-problem follows:

\begin{theorem}\thlabel{main1}
Let $\mathfrak{g}$ be a Leibniz algebra, $E$ a vector space that
contains $\mathfrak{g}$ as a subspace and $V$ a complement of
$\mathfrak{g}$ in $E$. Then:

$(1)$ $\equiv$ is an equivalence relation on the set ${\mathcal
L}{\mathcal Z} (\mathfrak{g}, V)$ of all Leibniz extending
structures of $\mathfrak{g}$ through $V$. If we denote ${\mathcal
H}{\mathcal L}^{2}_{\mathfrak{g}} \, (V, \, \mathfrak{g} ) :=
{\mathcal L} {\mathcal Z} (\mathfrak{g}, V)/ \equiv $, then the
map
$$
{\mathcal H}{\mathcal L}^{2}_{\mathfrak{g}} \, (V, \, \mathfrak{g}
) \to {\rm ExtdL} \, (E, \mathfrak{g}), \qquad
\overline{(\triangleleft, \triangleright, \leftharpoonup, \,
\rightharpoonup, f, \{-, \, -\})} \mapsto \bigl(\mathfrak{g}
\ltimes V, \, [ -  , \, - ] \bigl)
$$
is bijective, where $\overline{(\triangleleft, \triangleright,
\leftharpoonup, \, \rightharpoonup, f, \{-, \, -\})}$ is the
equivalence class of $(\triangleleft, \triangleright,
\leftharpoonup, \, \rightharpoonup, f, \{-, \, -\})$ via $\equiv$.

$(2)$ $\approx$ is an equivalence relation on the set ${\mathcal
L} {\mathcal Z} (\mathfrak{g}, V)$ of all Leibniz extending
structures of $\mathfrak{g}$ through $V$. If we denote ${\mathcal
H}{\mathcal L}^{2} \, (V, \, \mathfrak{g} ) := {\mathcal
L}{\mathcal Z} (\mathfrak{g}, V)/ \approx $, then the map
$$
{\mathcal H}{\mathcal L}^{2} \, (V, \, \mathfrak{g} ) \to {\rm
ExtdL}' \, (E, \mathfrak{g}), \qquad
\overline{\overline{(\triangleleft, \triangleright,
\leftharpoonup, \, \rightharpoonup, f, \{-, \, -\})}} \mapsto
\bigl(\mathfrak{g} \ltimes V, \, [ -  , \, - ] \bigl)
$$
is bijective, where $\overline{\overline{(\triangleleft,
\triangleright, \leftharpoonup, \, \rightharpoonup, f, \{-, \,
-\})}}$ is the equivalence class of $(\triangleleft,
\triangleright, \leftharpoonup, \, \rightharpoonup,  f, \{-, \,
-\})$ via $\approx$.
\end{theorem}

\section{Flag extending structures of Leibniz algebras. Examples}\selabel{exemple}

After we have provided a theoretical answer to the ES problem in
\thref{main1}, we are left to compute the classifying object
${\mathcal H} {\mathcal L}^{2}_{\mathfrak{g}} \, (V, \,
\mathfrak{g} )$ for a given Leibniz algebra $\mathfrak{g}$ that is
a subspace in a vector space $E$ with a complement $V$ and then to
describe all Leibniz algebra structures on $E$ which extend the
one of $\mathfrak{g}$. In this section we propose an algorithm to
tackle the problem for a large class of such structures.

\begin{definition} \delabel{flagex}
Let $\mathfrak{g}$ be a Leibniz algebra and $E$ a vector space
containing $\mathfrak{g}$ as a subspace. A Leibniz algebra
structure on $E$ is called a \emph{flag extending structure} of
$\mathfrak{g}$ if there exists a finite chain of Leibniz
subalgebras of $E$
\begin{equation} \eqlabel{lant}
\mathfrak{g} = E_0 \subset E_1 \subset \cdots \subset E_m = E
\end{equation}
such that $E_i$ has codimension $1$ in $E_{i+1}$, for all $i = 0,
\cdots, m-1$.
\end{definition}

As an easy consequence of \deref{flagex} we have that ${\rm dim}_k
(V) = m$, where $V$ is the complement of $\mathfrak{g}$ in $E$. In
what follows we will provide a way of describing all flag
extending structures of $\mathfrak{g}$ to $E$ in a recursive
manner which relies on the first step, namely $m = 1$. Therefore,
we start by describing and classifying all unified products
$\mathfrak{g} \ltimes V_1$, for a $1$-dimensional vector space
$V_1$. This procedure can be iterated by replacing the initial
Leibniz algebra $\mathfrak{g}$ with a unified product
$\mathfrak{g} \ltimes V_1$ obtained in the previous step. After
$m$ steps we arrive at the description of all flag extending
structures of $\mathfrak{g}$ to $E$. We start by introducing the
following two concepts:

\begin{definition} \delabel{tehnicaa}
Let $\mathfrak{g}$ be a Leibniz algebra. A \emph{flag datum of the
first kind} of $\mathfrak{g}$ is a $5$-tuple $(g_{0}, \, \alpha,
\, \lambda,\, D,\, \Delta)$, where $g_{0} \in \mathfrak{g}$,
$\alpha \in k$, $\lambda: \mathfrak{g} \to k$, $D$ and $\Delta:
\mathfrak{g} \to \mathfrak{g}$ are linear maps satisfying the
following compatibilities for any $g$, $h\in \mathfrak{g}$:
\begin{enumerate}
\item[(F1)] $\lambda\bigl([g, \, h]\bigl) = 0$, \, \,
$\lambda\bigl(D(g)\bigl) + \, \alpha \, \lambda(g) = 0$, \, \,
$\lambda\bigl(\Delta (g)\bigl) = 0$;

\item[(F2)] $D(g_0) = - \, \alpha \, g_{0}$, \,\, $\lambda (g_0) =
- \, \alpha^{2}$, \,\,  $\alpha \, \Delta(g) = -[g, \, g_0]$;

\item[(F3)] $[g, \, \Delta(h)] + [g, \, D(h)] = - \lambda(h) \,
\Delta(g)$;

\item[(F4)] $D^{2}(g) + D\bigl(\Delta(g)\bigl) \, = - \lambda(g)
\, g_{0}$;

\item[(F5)] $D^{2}(g) + \Delta \bigl(D(g)\bigl) \, = \alpha \,
D(g) + [g_{0},\, g] - 2 \, \lambda(g) \, g_{0}$;

\item[(F6)] $\Delta \bigl([g, \, h]\bigl) \, = [\Delta(g), \, h] +
[g, \, \Delta(h)]$;

\item[(F7)] $D\bigl([g, \, h]\bigl) \, = [D(g), \, h] - [D(h), \,
g] + \lambda(g) D(h) - \lambda(h) D(g)$
\end{enumerate}
We denote by ${\mathcal F}_{1} \, (\mathfrak{g})$ the set of all
flag datums of the first kind of $\mathfrak{g}$.
\end{definition}

\begin{definition} \delabel{tehnicab}
Let $\mathfrak{g}$ be a Leibniz algebra. A \emph{flag datum of the
second kind} of $\mathfrak{g}$ is a quadruple $(g_{0}, \, \nu,\,
D,\, \Delta)$, where $g_{0} \in \mathfrak{g}$, $\nu: \mathfrak{g}
\to k$, $\nu \neq 0$ is a non-trivial map, $D$ and $\Delta:
\mathfrak{g} \to \mathfrak{g}$ are linear maps satisfying the
following compatibilities for any $g$, $h\in \mathfrak{g}$:
\begin{enumerate}
\item[(G1)] $\nu\bigl([g,\, h]\bigl) = 0$, \,\, $[g, \, g_0] = 0$,
\,\, $D(g_0) = 0$, \,\, $\nu (g_0) = 0$;

\item[(G2)] $[g, \, \Delta(h)] + [g, \, D(h)] = 0$, \,\,\,\,
$\nu\bigl(D(g)\bigl) + \, \nu\bigl(\Delta (g)\bigl) = 0$;

\item[(G3)] $D^{2}(g) + D\bigl(\Delta (g)\bigl) \, = 0$, \, \,
$D^{2}(g) + \Delta\bigl(D(g)\bigl) \, = [g_{0}, \, g] + 2 \,
\nu(g) \, g_{0}$;

\item[(G4)] $\Delta\bigl([g, \, h]\bigl) \, = [\Delta(g), \, h] +
[g, \, \Delta(h)] + \nu(h) \Delta(g) + \nu(g) D(h)$;

\item[(G5)] $D\bigl([g, \, h]\bigl) \, = [D(g), \, h] - [D(h), \,
g] - \nu(g) D(h) + \nu(h) D(g)$;
\end{enumerate}
\end{definition}

We denote by ${\mathcal F}_{2} \, (\mathfrak{g})$ the set of all
flag datums of the second kind of $\mathfrak{g}$ and by ${\mathcal
F} \, (\mathfrak{g}) := {\mathcal F}_{1} \, (\mathfrak{g}) \sqcup
{\mathcal F}_{2} \, (\mathfrak{g})$, the disjoint union of the two
sets. The elements of ${\mathcal F} \, (\mathfrak{g})$ will be
called \emph{flag datums} of $\mathfrak{g}$. ${\mathcal F} \,
(\mathfrak{g})$ contains the space of pointed double derivations
${\mathcal D} \, (\mathfrak{g})$ of $\mathfrak{g}$ via the
canonical embedding: ${\mathcal D} (\mathfrak{g}) \hookrightarrow
{\mathcal F}_1 \, (\mathfrak{g})$, $(g_0, \, D, \, \Delta) \mapsto
(g_0, \, 0, \, 0, \, D, \, \Delta)$. The next proposition shows
that the space ${\mathcal F} \, (\mathfrak{g})$ is the counterpart
for Leibniz algebras of what we have called in \cite[Definition
4.2]{am-2013} the space of twisted derivations of a Lie algebra:

\begin{proposition}\prlabel{unifdim1}
Let $\mathfrak{g}$ be a Leibniz algebra and $V$ a vector space of
dimension $1$ with a basis $\{x\}$. Then there exists a bijection
between the set ${\mathcal L} {\mathcal Z} \, (\mathfrak{g}, V)$
of all Leibniz extending structures of $\mathfrak{g}$ through $V$
and ${\mathcal F} \, (\mathfrak{g}) = {\mathcal F}_{1} \,
(\mathfrak{g}) \sqcup {\mathcal F}_{2} \, (\mathfrak{g})$.

Under the above bijective correspondence the Leibniz extending
structure $\Omega(\mathfrak{g}, V)  = \bigl(\triangleleft, \,
\triangleright, \leftharpoonup,\, \rightharpoonup, \, f, \{-, -\}
\bigl)$ corresponding to $(g_0, \, \alpha, \, \lambda, \, D, \, \Delta)
\in {\mathcal F}_{1} \, (\mathfrak{g})$ is given by:
\begin{eqnarray}
x \triangleleft g &=& \lambda (g) x, \quad \,\,\,\,\,\,\,\,
x \triangleright g = D(g), \quad f (x, x) = g_0 \eqlabel{extenddim1.1a} \\
g \leftharpoonup x &=& \Delta(g), \qquad g \rightharpoonup x = 0,
\qquad \,\,\,\,\, \{x, \, x\} = \alpha \, x
\eqlabel{extenddim1.1b}
\end{eqnarray}
while the Leibniz extending structure $\Omega(\mathfrak{g}, V)  =
\bigl(\triangleleft, \, \triangleright, \leftharpoonup,\,
\rightharpoonup, \, f, \{-, -\} \bigl)$ corresponding to $(g_0, \,
\nu, \, D, \, \Delta) \in {\mathcal F}_{2} \, (\mathfrak{g})$ is given
by:
\begin{eqnarray}
x \triangleleft g &=& - \nu (g) x, \quad \,\,\, x \triangleright g
= D(g), \quad f (x, x) = g_0 \eqlabel{extenddim1.2a}\\
g \leftharpoonup x &=& \Delta (g), \qquad g \rightharpoonup x = \nu (g)
x, \quad \{x, \, x\} = 0 \eqlabel{extenddim1.2b}
\end{eqnarray}
for all $g \in \mathfrak{g}$.
\end{proposition}

\begin{proof}
We have to compute the set of all bilinear maps $ \triangleleft :
V \times \mathfrak{g} \to V$, $\triangleright : V \times
\mathfrak{g} \to \mathfrak{g}$, $\leftharpoonup : \mathfrak{g}
\times V \to \mathfrak{g}$, $\rightharpoonup : \mathfrak{g} \times
V \to V$, $f: V\times V \to \mathfrak{g}$ and $\{-, \, -\} :
V\times V \to V$ satisfying the compatibility conditions
$(L1)-(L14)$ of \thref{1}. Since $V$ has dimension $1$ there
exists a bijection between the set of all bilinear maps
$\triangleleft : V \times \mathfrak{g} \to V$ and the set of all
linear maps $\lambda : \mathfrak{g} \to k$ and the bijection is
given such that the action $\triangleleft : V \times \mathfrak{g}
\to V$ associated to $\lambda$ is given by the formula: $ x
\triangleleft g := \lambda (g) x$, for all $g\in \mathfrak{g}$. In
the same manner, the action $\rightharpoonup: \mathfrak{g} \times
V \to V$ is uniquely determined by a linear map $\nu: \mathfrak{g}
\to k$ such that $g \rightharpoonup x = \nu(g) x$, for all $g \in
\mathfrak{g}$. Similarly, the bilinear maps $\triangleright :
V\times \mathfrak{g} \to \mathfrak{g}$, $\leftharpoonup:
\mathfrak{g} \times V \to \mathfrak{g}$ are uniquely implemented
by linear maps $D = D_{\triangleright} : \mathfrak{g} \to
\mathfrak{g}$ respectively $\Delta = \Delta_{\leftharpoonup} : \mathfrak{g}
\to \mathfrak{g}$ via the formulas: $x \triangleright g := D(g)$
and $g \leftharpoonup x := \Delta (g)$, for all  $g\in \mathfrak{g}$.
Finally, any bilinear map $f : V \times V \to \mathfrak{g}$ is
uniquely implemented by an element $g_0 \in \mathfrak{g}$ such
that $f (x, x) = g_0$ and any bracket $\{-, \, -\} : V\times V \to
V$ is uniquely determined by a scalar $\alpha \in k$ such that
$\{x, \, x\} = \alpha x$.

Now, the compatibility condition $(L9)$ is equivalent to $\alpha
\, \nu(g) = 0$, for all $g\in \mathfrak{g}$, while $(L14)$ gives
$\nu(g)\, \bigl(\nu(h) + \lambda(h)\bigl) = 0$, for all $g$, $h
\in \mathfrak{g}$. If $\nu = 0$, the trivial map on
$\mathfrak{g}$, then $(L9)$ and $(L14)$ are trivially fulfilled
and the rest of the compatibility conditions of \thref{1} came
down to $(F1)$-$(F7)$ from the definition of ${\mathcal F}_{1} \,
(\mathfrak{g})$. This can be proved by a routinely computation:
for instance, axiom $(L1)$ is equivalent to $\lambda \bigl( [g, \,
h ] \bigl) = 0$ while axiom $(L2)$ is equivalent to $(F7)$.
Otherwise, if $\nu \neq 0$ implies that $\alpha = 0$ and $\lambda
= - \nu$. Based on this, it is straightforward to see that the
compatibility conditions $(L1)-(L14)$ of \thref{1} are equivalent
to $(G1)-(G5)$ from the definition of ${\mathcal F}_{2} \,
(\mathfrak{g})$.
\end{proof}

Let $(g_0, \, \alpha,\, \lambda, \, D, \, \Delta) \in {\mathcal F}_{1}
\, (\mathfrak{g})$. The unified product $\mathfrak{g}
\ltimes_{(g_0, \, \alpha,\, \lambda, \, D, \, \Delta)} V$ associated to
the Leibniz extending structure given by \equref{extenddim1.1a} -
\equref{extenddim1.1b} will be denoted by $\mathfrak{g}_{1} \, ( x
\, | \, (g_0, \,\alpha,\, \lambda, \, D, \, \Delta))$ and has the
bracket defined by:
\begin{eqnarray}
\left[(g, 0), \, (h, 0)\right] &=& ([g, h], \, 0), \quad \, [(g,
0), \, (0, x)] = (\Delta(g), \, 0) \eqlabel{primuunif1} \\
\left[(0, x), \, (0, x)\right] &=& (g_0, \, \alpha \, x), \quad
\,\,\, [(0, x), \, (g, 0)] = (D(g), \, \lambda(g) x)
\eqlabel{primuunif2}
\end{eqnarray}
for all $g$, $h\in \mathfrak{g}$. On the other hand, for $(g_0, \,
\nu, \, D, \, \Delta) \in {\mathcal F}_{2} \, (\mathfrak{g})$, the
unified product $\mathfrak{g} \ltimes_{(g_0, \, \nu, \, D, \, \Delta)}
V$ associated to the Leibniz extending structure given by
\equref{extenddim1.2a}-\equref{extenddim1.2b} will be denoted by
$\mathfrak{g}_{2} \, ( x \, | \, (g_0, \, \nu, \, D, \, \Delta))$ and
has the bracket defined by:
\begin{eqnarray}
\left[(g, 0), \, (h, 0)\right] &=& ([g, h], \, 0), \quad [(g, 0),
\, (0, x)] = (\Delta (g), \, \nu(g) x) \eqlabel{primuunif11} \\
\left[(0, x), \, (0, x) \right] &=& (g_0, \, 0), \quad
\,\,\,\,\,\,\, [(0, x), \, (g, 0)] = (D(g), \, - \nu(g) x)
\eqlabel{primuunif12}
\end{eqnarray}
for all $g$, $h\in \mathfrak{g}$. Thus, we have obtained the
following:

\begin{corollary}\colabel{descflagnecl}
Any Leibniz algebra that contains a given Leibniz algebra
$\mathfrak{g}$ as a subalgebra of codimension $1$ is isomorphic to
a Leibniz algebra of type $\mathfrak{g}_{1} \, ( x \, | \, (g_0,
\,\alpha,\, \lambda, \, D, \, \Delta))$, for some $(g_0,
\,\alpha,\, \lambda, \, D, \, \Delta) \in {\mathcal F}_1 \,
(\mathfrak{g})$ or to a Leibniz algebra of type $\mathfrak{g}_{2}
\, ( x \, | \, (g_0, \, \nu, \, D, \, \Delta))$, for some $(g_0,
\, \nu, \, D, \, \Delta) \in {\mathcal F}_2 \, (\mathfrak{g})$.
\end{corollary}

Now, we shall classify these Leibniz algebras up to an isomorphism
that stabilizes $\mathfrak{g}$, i.e. we give the first explicit
classification result for the ES-problem. This is the key step in
the classification of all flag extending structures of
$\mathfrak{g}$.

\begin{theorem}\thlabel{clasdim1}
Let $\mathfrak{g}$ be a Leibniz algebra of codimension $1$ in the
vector space $E$. Then:
$$
{\rm ExtdL} \, (E, \mathfrak{g}) \, \cong \, {\mathcal H}
{\mathcal L}^{2}_{\mathfrak{g}} (k, \mathfrak{g} ) \, \cong \,
({\mathcal F}_1 \, (\mathfrak{g}) /\equiv_1) \sqcup ({\mathcal
F}_2 \, (\mathfrak{g}) /\equiv_2), \qquad {\rm where:}
$$
$\equiv_1$ is the equivalence relation on the set ${\mathcal F}_1
\, (\mathfrak{g})$ defined as follows: $(g_{0}, \, \alpha, \,
\lambda,\, D,\, \Delta) \equiv_1 (g'_{0}, \, \alpha', \, \lambda', \,
D', \, \Delta')$ if and only if $\lambda = \lambda'$ and there exists a
pair $(q, G) \in k^* \times \mathfrak{g}$ such that for any $g \in
\mathfrak{g}$:
\begin{eqnarray}
g_0 &=& q^2 \, g'_0 + [G, G] + q \, D' (G) + q \, E' (G) - q \,
\alpha' G - \lambda' (G) G \eqlabel{lzeci1} \\
\alpha &=& q\, \alpha' + \lambda'(G) \eqlabel{lzeci2} \\
D(g) &=& q \, D'(g) + [G, g] - \lambda'(g) G \eqlabel{lzeci3} \\
\Delta(g)  &=& q \, \Delta'(g) + [g, G] \eqlabel{lzeci4}
\end{eqnarray}
$\equiv_2$ is the equivalence relation on ${\mathcal F}_2 \,
(\mathfrak{g})$ given by: $(g_{0}, \, \nu, \, D,\, \Delta) \equiv_2
(g'_{0}, \, \nu',\, D', \, \Delta')$ if and only if $\nu = \nu'$ and
there exists a pair $(q, G) \in k^* \times \mathfrak{g}$ such that
for any $g \in \mathfrak{g}$:
\begin{eqnarray}
g_0 &=& q^2 \, g'_0 + [G, G] + q \, D' (G) + q \, E' (G) \eqlabel{lzeci1a} \\
D(g) &=& q \, D'(g) + [G, g] + \nu'(g) \, G \eqlabel{lzeci2a} \\
\Delta(g)  &=& q \, \Delta'(g) + [g, G] - \nu'(g) \, G \eqlabel{lzeci3a}
\end{eqnarray}
The bijection between $({\mathcal F}_1 \, (\mathfrak{g})
/\equiv_1) \sqcup ({\mathcal F}_2 \, (\mathfrak{g}) /\equiv_2)$
and ${\rm ExtdL} \, (E, \mathfrak{g})$ is given by: \footnote{As
usual we denote by $\overline{y}^i$ the equivalence class of $y$
via the relation $\equiv_i$, $i = 1$, $2$.}
$$
\overline{(g_{0}, \, \alpha, \, \lambda,\, D,\, \Delta)}^1 \mapsto
\mathfrak{g}_{1} \, ( x \, | \, (g_0, \, \alpha, \, \lambda, \, D,
\, \Delta)) \,\,\, {\rm and} \,\,\, \overline{(g_{0}, \, \nu, \, D,\,
\Delta)}^2 \mapsto \mathfrak{g}_{2} \, ( x \, | \, (g_0, \, \nu, \, D,
\, \Delta))
$$
\end{theorem}

\begin{proof}
The proof relies on \prref{unifdim1} and \thref{main1}. Let $V$ be
a complement of $\mathfrak{g}$ in $E$ having $\{x\}$ as a basis.
Since ${\rm dim}_k (V) = 1$, any linear map $r: V \to
\mathfrak{g}$ is uniquely determined by an element $G \in
\mathfrak{g}$ such that $r(x) = G$, where $\{x\}$ is a basis in
$V$. On the other hand, any automorphism $v$ of $V$ is uniquely
determined by a non-zero scalar $q\in k^*$ such such $v (x) = q
x$. Based on these facts, a little computation shows that the
compatibility conditions from \deref{echiaa}, imposed for the
Leibniz extending structures
\equref{extenddim1.1a}-\equref{extenddim1.1b} and respectively
\equref{extenddim1.2a}-\equref{extenddim1.2b}, take precisely the
form given in the statement of the theorem. We should mention here
that a Leibniz extending structure given by \equref{extenddim1.1a}
- \equref{extenddim1.1b} is never equivalent in the sense of
\deref{echiaa} to a Leibniz extending structure given by
\equref{extenddim1.2a}-\equref{extenddim1.2b}, thanks to the
compatibility condition (ML1) of \leref{morfismunif}. Therefore,
we obtain the disjoint union from the statement and the proof is
finished.
\end{proof}

\begin{remark}\relabel{alalatfl}
In the context of \thref{clasdim1} we also have that
$$
{\rm ExtdL'} \, (E, \mathfrak{g}) \, \cong \, {\mathcal H}
{\mathcal L}^{2}(k, \mathfrak{g} ) \, \cong \, ({\mathcal F}_1 \,
(\mathfrak{g}) /\approx_1) \sqcup ({\mathcal F}_2 \,
(\mathfrak{g}) /\approx_2), \qquad {\rm where:}
$$
$\approx_i$ is the following relation on ${\mathcal F}_i \,
(\mathfrak{g})$: $(g_{0}, \, \alpha, \, \lambda,\, D,\, \Delta)
\approx_1 (g'_{0}, \, \alpha', \, \lambda', \, D', \, \Delta')$ if
and only if $\lambda = \lambda'$ and there exists $G \in
\mathfrak{g}$ such that relations \equref{lzeci1}-\equref{lzeci4}
hold for $q = 1$ and respectively $(g_{0}, \, \nu, \, D,\, \Delta)
\approx_2 (g'_{0}, \, \nu',\, D', \, \Delta')$ if and only if $\nu
= \nu'$ and there exists $ G \in \mathfrak{g}$ such that
\equref{lzeci1a}-\equref{lzeci3a} hold for $q=1$.
\end{remark}

\thref{clasdim1} takes a simplified form for perfect Leibniz
algebras. Indeed, let $\mathfrak{g}$ be a perfect Leibniz algebra,
i.e. $\mathfrak{g}$ is generated as a vector space by all brackets
$[x, y]$. Then (G1) shows that ${\mathcal F}_2 \, (\mathfrak{g})$
is the empty set since, by definition, an element $\nu$ of a
quadruple $(g_{0}, \, \nu,\, D,\, \Delta) \in {\mathcal F}_2 \,
(\mathfrak{g})$ is a non-trivial map. Thus, we have that
${\mathcal F} \, (\mathfrak{g}) = {\mathcal F}_1 \,
(\mathfrak{g})$. Let now $(g_{0}, \, \alpha, \, \lambda,\, D,\,
\Delta) \in {\mathcal F}_1 \, (\mathfrak{g})$; it follows from
(F1) and (F2) that $\lambda = 0$, the trivial map, and $\alpha =
0$. Furthermore, we can easily see that for a perfect Leibniz
algebra $\mathfrak{g}$, ${\mathcal F} \, (\mathfrak{g})$
identifies with the set of all triples $(g_{0}, \, D, \, \Delta)$,
where $g_{0} \in \mathfrak{g}$, $D$, $\Delta: \mathfrak{g} \to
\mathfrak{g}$ are linear maps satisfying the compatibilities
\equref{flagper1}-\equref{flagper4}, that is ${\mathcal F} \,
(\mathfrak{g}) \cong {\mathcal D} (\mathfrak{g})$, where
${\mathcal D} (\mathfrak{g})$ is the space of all pointed double
derivations of $\mathfrak{g}$ as defined in \deref{dder}.

Two pointed double derivations $(g_{0}, \, D, \, \Delta)$ and
$(g'_{0}, \, D', \, \Delta')$ are equivalent and we write $(g_{0},
\, D, \, \Delta) \equiv (g'_{0}, \, D', \, \Delta')$ if and only
if there exists a pair $(q, G) \in k^* \times \mathfrak{g}$ such
that:
\begin{eqnarray}
g_0 &=& q^2 \, g'_0 + [G, G] + q \, D' (G) + q \, \Delta' (G) \eqlabel{lzeci1per} \\
D - q \, D' &=&  [G, - ], \qquad \Delta - q \, \Delta' = [ - , G]
\eqlabel{lzeci4per}
\end{eqnarray}
On the other hand, two pointed double derivations $(g_{0}, \, D,
\, \Delta)$ and $(g'_{0}, \, D', \, \Delta')$ are cohomologous and
we write $(g_{0}, \, D, \, \Delta) \approx (g'_{0}, \, D', \,
\Delta')$ if and only if there exists $ G \in \mathfrak{g}$ such
that:
\begin{eqnarray}
g_0 &=& g'_0 + [G, G] +  D' (G) +  \Delta' (G) \eqlabel{lzeci1per2} \\
D -  \, D' &=&  [G, - ], \qquad \Delta - \, \Delta' = [ - , G]
\eqlabel{lzeci4per2}
\end{eqnarray}

Taking into account the unified product defined by
\equref{primuunif1}-\equref{primuunif2} we obtain:

\begin{corollary}\colabel{clasdim1perf}
Let $\mathfrak{g}$ be a perfect Leibniz algebra having $\{e_i \, |
\, i\in I \}$ as a basis. Then any Leibniz algebra $E$ containing
$\mathfrak{g}$ as a subalgebra of codimension $1$ has the bracket
$[-, -]_E $ defined on the basis $\{x, \, e_i \, | \, i\in I \}$
by:
\begin{eqnarray*}
&& \left[e_i, \, e_j\right]_E := [e_i, \, e_j], \quad [x, \, x]_E
:= g_0, \quad \left[e_i, \, x\right]_E := \Delta(e_i), \quad \left[x,
\, e_i\right]_E := D(e_i)
\end{eqnarray*}
for all $(g_{0}, \, D, \, \Delta) \in {\mathcal D} (\mathfrak{g})
$. Furthermore, ${\mathcal H} {\mathcal L}^{2}_{\mathfrak{g}} (k,
\mathfrak{g} ) \cong {\mathcal D} \, (\mathfrak{g}) /\equiv$ and
${\mathcal H} {\mathcal L}^{2}(k, \mathfrak{g} ) \cong {\mathcal
D} \, (\mathfrak{g}) / \approx $, where $\equiv$ (resp. $\approx$)
is the relation defined by \equref{lzeci1per}-\equref{lzeci4per}
(resp. \equref{lzeci1per2}-\equref{lzeci4per2}).
\end{corollary}

On the other hand, we have the following result for abelian
Leibniz algebras:

\begin{example} \exlabel{lbzabelian}
Let $\mathfrak{g}$ be a vector space with $\{e_i \, | \, i\in I
\}$ as a basis viewed as an abelian Leibniz algebra. Then, there
exist three families of Leibniz algebras that contain
$\mathfrak{g}$ as a subalgebra of codimension $1$. They have $\{x,
\, e_i \, | \, i\in I \}$ as a basis and the bracket given for any
$i\in I$ as follows:
\begin{eqnarray*}
&\mathfrak{g}_{11}^{(g_0, \, D, \, \Delta)}:& \qquad \left[e_i, \,
e_j\right] = 0, \quad  \left[e_i, \, x \right] = \Delta (e_i), \quad
\left[x, \, x\right] = g_0, \quad  \left[x, \, e_i\right] = D(e_i)
\end{eqnarray*}
for all triples $(g_0, D, \Delta) \in \mathfrak{g} \times {\rm
Hom}_k (\mathfrak{g}, \mathfrak{g})^2$ such that $g_0 \in {\rm
Ker} (D)$ and $ D\circ \Delta =  \Delta \circ D = - D^2$. The
Leibniz algebra $\mathfrak{g}_{11}^{(g_0, \, D, \, \Delta)}$ is
the unified product associated to the flag datum of the first kind
$(g_{0}, \, \alpha, \, \lambda,\, D,\, \Delta)$ for which $\alpha
:= 0$ and $\lambda := 0$.

The second family of Leibniz algebras has the bracket given as
follows:
\begin{eqnarray*}
&\mathfrak{g}_{12}^{(u, \, h_0, \, \lambda)}:& \qquad \left[e_i,
\, e_j\right] = 0, \quad  \left[e_i, \, x \right] = 0, \quad
\left[x, \, e_i\right] =  u \, \lambda(e_i) \, h_0 + \lambda (e_i) \, x \\
&& \qquad \left[x, \, x\right] = - u^2 \, \lambda(h_0) \, h_0 -
u\, \lambda(h_0) \, x
\end{eqnarray*}
for all triples $(u, \, h_0, \, \lambda) \in k^* \times
\mathfrak{g} \times {\rm Hom}_k (\mathfrak{g}, k)$ such that
$\lambda \neq 0$. The Leibniz algebra $\mathfrak{g}_{12}^{(u, \,
h_0, \, \lambda)}$ is the unified product associated to the flag
datum of the first kind $(g_{0}, \, \alpha, \, \lambda,\, D,\, \Delta)$
for which $g_0 := u^2 \lambda(h_0) \, h_0$, $\alpha := - u
\lambda(h_0)$, $\Delta := 0$ and $D(g) := u \lambda(g) \, h_0$, for all
$g \in \mathfrak{g}$.

Finally, for the last family of Leibniz algebras the bracket is
given as follows:
\begin{eqnarray*}
&\mathfrak{g}_{2}^{(u, \, g_0, \, h_0, \, \nu)}:& \quad \,
\left[e_i, \, e_j\right] = 0, \quad  \left[e_i, \, x \right] = -
\left[x, \, e_i\right] = - u \, \nu (e_i) \, h_0 + \nu(e_i) \, x,
\quad \left[x, \, x\right] = g_0
\end{eqnarray*}
for all $(u, \, g_0, \, h_0, \, \nu) \in k^* \times \mathfrak{g}^2
\times {\rm Hom}_k (\mathfrak{g}, k)$ such that: $2 g_0 = 0$,
$\nu(g_0) = 0$ and $\nu \neq 0$. The Leibniz algebra
$\mathfrak{g}_{2}^{(u, \, g_0, \, h_0, \, \nu)}$ is the unified
product associated to the flag datum of the second kind $(g_{0},
\, \nu,\, D,\, \Delta)$ for which $D(g) := u \, \nu (g) h_0$ and $\Delta(g)
 := - u \, \nu (g) h_0$, for all $g \in \mathfrak{g}$.

The above results are obtained by a straightforward computation
which relies on the explicit description of the set ${\mathcal F}
\, (\mathfrak{g})$ for an abelian Leibniz algebra. For instance,
axiom (F3) from the flag datum of first kind, for an abelian
Leibniz algebra, takes the form $\lambda (h) \Delta(g) = 0$, for
all $g$, $h \in \mathfrak{g}$. Therefore, we have to consider two
cases in order to describe the set ${\mathcal F}_1 \,
(\mathfrak{g})$: the first one corresponds to $\lambda = 0$, while
for the second one we have $\lambda \neq 0$. The corresponding
unified products are the ones given by the first two families of
Leibniz algebras. In order to describe the set ${\mathcal F}_2 \,
(\mathfrak{g})$ we mention that the condition $2 g_0 = 0$ is
derived from axiom (G2). In this case the corresponding unified
product is the one given by the last family of Leibniz algebras.
\end{example}

Next we provide some explicit examples. First, we prove that any
Leibniz algebra which contains a semisimple Lie algebra
$\mathfrak{g}$ as a subalgebra  of codimension $1$ is in fact a
Lie algebra and the classifying objects ${\mathcal
H}^{2}_{\mathfrak{g}} (k, \mathfrak{g} )$ and ${\mathcal H}^{2}
(k, \mathfrak{g} )$ are both singletons.

\begin{example} \exlabel{liesemis}
Let $\mathfrak{g}$ be a semisimple Lie algebra of codimension $1$
in the vector space $E$ and $\{e_i \, | \, i = 1, \cdots, n \}$ a
basis of $\mathfrak{g}$. Then any Leibniz algebra structure on $E$
that contains $\mathfrak{g}$ as a subalgebra is isomorphic to the
Lie algebra having $\{x, \, e_i \, | \, i = 1, \cdots, n \}$ as a
basis and the bracket $[-, -]_E $ defined by for any $i = 1,
\cdots, n$ by:
\begin{equation*} \eqlabel{semisimplulbz}
[e_i, \, e_j]_E := [e_i, e_j], \quad [x, \, e_i]_E = - [e_i, \,
x]_E := [h_0, e_i], \quad [x, \, x] =0
\end{equation*}
for some $h_0 \in \mathfrak{g}$. Furthermore, ${\mathcal H}
{\mathcal L}^{2}_{\mathfrak{g}} (k, \mathfrak{g} ) = {\mathcal H}
{\mathcal L}^{2} (k, \mathfrak{g} ) = 0$.

Indeed, we apply \coref{clasdim1perf} taking into account that any
semisimple Lie algebra $\mathfrak{g}$ is perfect, ${\rm
Inn}(\mathfrak{g}) = {\rm Der} (\mathfrak{g})$ and
$Z(\mathfrak{g}) = 0$. Let $(g_{0}, \, D, \, \Delta ) \in
{\mathcal F} \, (\mathfrak{g})$ be a flag datum of $\mathfrak{g}$;
then, since $\mathfrak{g}$ has a trivial center we obtain from
\equref{flagper1} that $g_0 = 0$. Moreover, as $\mathfrak{g}$ is
perfect it follows again from \equref{flagper1} that $\Delta = -
D$. Thus, ${\mathcal F} \, (\mathfrak{g}) = {\rm Der}
(\mathfrak{g})$. Since $\mathfrak{g}$ is semisimple any derivation
$D \in {\rm Der} (\mathfrak{g})$ is inner, i.e. there exists $h_0
\in \mathfrak{g}$ such that $D = [h_0, -]$. Thus, the Leibniz
algebra $\mathfrak{g}_{1} \, ( x \, | \, (g_0, \,\alpha,\,
\lambda, \, D, \, \Delta)) = \mathfrak{g}_{1} \, ( x \, | \, D) $
defined by \equref{primuunif1} and \equref{primuunif2} takes the
form given in the statement. Moreover, two derivations $D = [h_0,
-]$ and $D' = [h'_0, -]$ are equivalent in the sense of
\equref{lzeci4per2} if and only if there exists $G \in
\mathfrak{g}$ such that $h_0 = h'_0 + G$, i.e. any two derivations
are cohomologous. This shows that ${\mathcal H} {\mathcal
L}^{2}(k, \mathfrak{g} )$ is a singleton having only $0$ as an
element and so is ${\mathcal H} {\mathcal L}^{2}_{\mathfrak{g}}
(k, \mathfrak{g} )$ being a quotient of it.
\end{example}

Now we will provide an explicit example which highlights the
efficiency of \thref{clasdim1}. More precisely, we will describe
all $4$-dimensional Leibniz algebras that contain a given
non-perfect $3$-dimensional Leibniz algebra $\mathfrak{g}$ as a
subalgebra. Then we will be able to compute the classifying object
${\rm ExtdL} \, (k^4, \mathfrak{g}) \cong {\mathcal H} {\mathcal
L}^{2}_{\mathfrak{g}} (k, \mathfrak{g} )$. The detailed
computations are rather long but straightforward and can be
provided upon request.

\begin{example} \exlabel{dim34}
Let $\mathfrak{g}$ be the $3$-dimensional Leibniz algebra with the
basis $\{e_{1}, e_{2}, e_{3}\}$ and the bracket defined by: $
[e_{1}, \, e_{3}] = e_{2}$, $[e_{3}, \, e_{3}] = e_{1}$.

Then, there exist four families of $4$-dimensional Leibniz
algebras which contain $\mathfrak{g}$ as a subalgebra: they have
$\{e_1, \, e_2, \, e_3, \, x \}$ as a basis and the bracket is
given as follows (the first three families of Leibniz algebras can
be defined over any field $k$ while in case of the fourth family
we need to distinguish between fields of characteristic $2$ and
those of characteristic different than $2$):

$(1)$ If ${\rm char} (k) \neq 2$ then the four families of Leibniz
algebras that contain $\mathfrak{g}$ as a subalgebra are the
following:
\begin{eqnarray*}
&\mathfrak{g}_{11}^{(b_1, \, b_2, \, c, \, d_1, \, d_2)}:& \qquad
\left[e_1, \, e_3\right] = e_2, \quad \left[e_3, \, e_3\right] =
e_1, \\
&& \qquad \left[e_1, \, x \right] = b_1 \, e_2,  \quad \left[e_3,
\, x \right] = b_1 \, e_1 + b_2 \, e_2, \\
&& \qquad \left[x, \, x\right] = b_1 d_1 \, e_1 + c\, e_2, \quad
\left[x, \, e_3\right] = d_1 \, e_1 + d_2 \, e_2
\end{eqnarray*}
for all $b_1$, $b_2$, $c$, $d_1$, $d_2 \in k$. The Leibniz algebra
$\mathfrak{g}_{11}^{(b_1, \, b_2, \, c, \, d_1, \, d_2)}$ is the
unified product associated to the flag datum of the first kind
$(g_{0}, \, \alpha, \, \lambda,\, D,\, \Delta)$ defined as
follows: $\alpha := 0$, $\lambda := 0$, $g_0 := b_1 d_1 \, e_1 + c
\, e_2$ and $D$, $\Delta$ are given by
$$
D :=
\left( \begin{array}{ccc} 0 & 0 & d_1  \\
0 & 0 & d_2 \\
0 & 0 & 0  \\
\end{array}\right)
\qquad \Delta :=
\left( \begin{array}{ccc} 0 & 0 & b_1  \\
b_1 & 0 & b_2 \\
0 & 0 & 0  \\
\end{array}\right)
$$
The second family of Leibniz algebras has the bracket given by:
\begin{eqnarray*}
&\mathfrak{g}_{12}^{(b_1, \, b_2, \, b_3, \, c, \, d)}:& \qquad
\left[e_1, \, e_3\right] = e_2, \quad \left[e_3, \, e_3\right] =
e_1, \quad \left[e_1, \, x \right] = 2 b_1 \, e_1 + b_2 \, e_2, \\
&& \qquad \left[e_2, \, x \right] = 3 b_1 \, e_2, \quad \left[e_3,
\, x \right] = b_2 \, e_1 + b_3 \, e_2 + b_1 \, e_3, \\
&& \qquad \left[x, \, x\right] = (2b_1 d + b_2^2 - b_1 b_3) \, e_1
+ c\, e_2, \quad \left[x, \, e_3\right] = b_2 \, e_1 + d \, e_2 -
b_1 \, e_3
\end{eqnarray*}
for all $b_1 \in k^*$ and $b_2$, $b_3$, $c$, $d \in k$. The
Leibniz algebra $\mathfrak{g}_{12}^{(b_1, \, b_2, \, b_3, \, c, \,
d)}$ is the unified product associated to the flag datum of the
first kind $(g_{0}, \, \alpha, \, \lambda,\, D,\, \Delta)$ defined
as follows: $\alpha := 0$, $\lambda := 0$, $g_0 := (2b_1 d + b_2^2
- b_1 b_3) \, e_1 + c \, e_2$ and $D$, $\Delta$ are given by:
$$
D :=
\left( \begin{array}{ccc} 0 & 0 & b_2  \\
0 & 0 & d \\
0 & 0 & -b_1  \\
\end{array}\right)
\qquad \Delta :=
\left( \begin{array}{ccc} 2b_1 & 0 & b_2  \\
b_2 & 3b_1 & b_3 \\
0 & 0 & b_1  \\
\end{array}\right)
$$
The third family of Leibniz algebras has the bracket given by:
\begin{eqnarray*}
&\mathfrak{g}_{13}^{(\alpha, \, \lambda_0, \, d_1, \, d_2)}:&
\qquad \left[e_1, \, e_3\right] = e_2, \quad \left[e_3, \,
e_3\right] = e_1, \quad
\left[e_1, \, x \right] = \alpha \, \lambda_0^{-1}\, e_2, \\
&& \qquad \left[e_3, \, x \right] = \alpha \, \lambda_0^{-1}\,
e_1, \quad \left[x, \, x\right] = \alpha \, \lambda_0^{-1}\, (d_1
\,  e_1 + d_2 \, e_2 - \alpha \, e_3) + \alpha \, x, \\
&& \qquad \left[x, \, e_3\right] = d_1 \, e_1 + d_2 \, e_2 -
\alpha \, e_3 + \lambda_0 \, x
\end{eqnarray*}
for all $\lambda_0 \in k^*$ and $\alpha$, $d_1$, $d_2 \in k$. The
Leibniz algebra $\mathfrak{g}_{13}^{(\alpha, \, \lambda_0, \, d_1,
\, d_2)}$ is the unified product associated to the flag datum of
the first kind $(g_{0}, \, \alpha, \, \lambda,\, D,\, \Delta)$
defined as follows: $\lambda  (e_1) = \lambda(e_2) := 0$, $\lambda
(e_3) := \lambda_0 \neq 0$, $g_0 := \alpha \, \lambda_0^{-1} d_1
\, e_1 + \alpha \, \lambda_0^{-1} d_2 \, e_2 - \alpha^2 \,
\lambda_0^{-1} e_3$ and $D$, $\Delta$ are given by
$$
D :=
\left( \begin{array}{ccc} 0 & 0 & d_1  \\
0 & 0 & d_2 \\
0 & 0 & - \alpha  \\
\end{array}\right)
\qquad \Delta :=
\left( \begin{array}{ccc} 0 & 0 & \alpha \, \lambda_0^{-1}  \\
\alpha \, \lambda_0^{-1} & 0 & 0 \\
0 & 0 & 0  \\
\end{array}\right)
$$
Finally, the last family of Leibniz algebras has the bracket
defined as follows:
\begin{eqnarray*}
&\mathfrak{g}_{21}^{(\nu_0, \, d_1, \, d_2, , d_3)}:& \qquad
\left[e_1, \, e_3\right] = e_2, \quad \left[e_3, \, e_3\right] =
e_1, \quad \left[e_1, \, x \right] = \nu_0^{-1} \, d_3 \, e_2, \\
&& \qquad \left[e_3, \, x \right] = (-d_1 + 2 \nu_0^{-1} d_3) e_1
- (d_2 - \nu_0^{-1} d_1 + \nu_0^{-2} d_3) e_2 - d_3 e_3 +
\nu_0 x, \\
&& \qquad \left[x, \, x\right] = \nu_0^{-2} \, d_3^2 \, e_1 +
(\nu_0^{-2} \, d_1 d_3 - \nu_0^{-3}
\, d_3^2) \, e_2, \\
&& \qquad \left[x, \, e_3\right] = d_1 \, e_1 + d_2 \, e_2 + d_3
\, e_3 - \nu_0 \, x
\end{eqnarray*}
for all $\nu_0 \in k^*$ and $d_1$, $d_2$, $d_3 \in k$. The Leibniz
algebra $\mathfrak{g}_{21}^{(\nu_0, \, d_1, \, d_2, \, d_3)}$ is
the unified product associated to the flag datum of the second
kind $(g_{0}, \, \nu,\, D,\, \Delta)$ defined as follows: $\nu
(e_1) = \nu(e_2) := 0$, $\nu (e_3) := \nu_0 \neq 0$, $g_0 :=
\nu_0^{-2} \, d_3^2 \, e_1 + (\nu_0^{-2} \, d_1 d_3 - \nu_0^{-3}
\, d_3^2) \, e_2 $ and $D$, $\Delta$ are given by
\begin{equation}\eqlabel{flagnasol}
D :=
\left( \begin{array}{ccc} 0 & 0 & d_1  \\
0 & 0 & d_2 \\
0 & 0 & d_3  \\
\end{array}\right)
\qquad \Delta :=
\left( \begin{array}{ccc} 0 & 0 & - d_1 + 2 \nu_0^{-1} d_3 \\
\nu_0^{-1} d_3 & 0 & - d_2 + \nu_0^{-1} d_1 - \nu_0^{-2} d_3 \\
0 & 0 & - d_3  \\
\end{array}\right)
\end{equation}

$(2)$ If ${\rm char} (k) = 2$, then the four families of Leibniz
algebras that contain $\mathfrak{g}$ as a subalgebra are the
following: $\mathfrak{g}_{11}^{(b_1, \, b_2, \, c, \, d_1, \,
d_2)}$, $\mathfrak{g}_{12}^{(b_1, \, b_2, \, b_3, \, c, \, d)}$,
$\mathfrak{g}_{13}^{(\alpha, \, \lambda_0, \, d_1, \, d_2)}$
defined above together with the family of Leibniz algebras defines
as follows:
\begin{eqnarray*}
&\mathfrak{g}_{22}^{(c, \, \nu_0, \, d_1, \, d_2, , d_3)}:& \quad
\left[e_1, \, e_3\right] = e_2, \quad \left[e_3, \, e_3\right] =
e_1, \quad \left[e_1, \, x \right] = \nu_0^{-1} \, d_3 \, e_2, \\
&& \quad \left[e_3, \, x \right] = -d_1 \, e_1 - (d_2 - \nu_0^{-1}
d_1 + \nu_0^{-2} d_3) e_2 - d_3 e_3 +
\nu_0 x, \\
&& \quad \left[x, \, x\right] = \nu_0^{-2} \, d_3^2 \, e_1 + c \,
e_2, \quad \left[x, \, e_3\right] = d_1 \, e_1 + d_2 \, e_2 + d_3
\, e_3 - \nu_0 \, x
\end{eqnarray*}
for all $\nu_0 \in k^*$ and $c$, $d_1$, $d_2$, $d_3 \in k$. The
Leibniz algebra $\mathfrak{g}_{22}^{(c, \, \nu_0, \, d_1, \, d_2,
\, d_3)}$ is the unified product associated to the flag datum of
the second kind $(g_{0}, \, \nu,\, D,\, \Delta)$ defined as
follows: $\nu (e_1) = \nu(e_2) := 0$, $\nu (e_3) := \nu_0 \neq 0$,
$g_0 := \nu_0^{-2} \, d_3^2 \, e_1 + c \, e_2 $ and $D$, $\Delta$
are given by \equref{flagnasol}.

The proof is a purely computational one and we will only indicate
the main steps. We start by computing ${\mathcal F}_1 \,
(\mathfrak{g})$. First, notice that a linear map $\lambda:
\mathfrak{g} \to k$ satisfies the first compatibility of (F1),
i.e. $\lambda ([g, h]) = 0$ if and only if $\lambda$ is given by
$\lambda (e_1) = \lambda (e_2) = 0$ and $\lambda (e_3) =
\lambda_0$, for some $\lambda_0 \in k$. For such a $\lambda$ we
can easily show that a pair $(D, \Delta)$ satisfies the
compatibilities (F6) and (F7) if and only if we have:
$$
D =
\left( \begin{array}{ccc} 0 & 0 & d_1  \\
0 & 0 & d_2 \\
0 & 0 & d_3  \\
\end{array}\right)
\qquad \Delta =
\left( \begin{array}{ccc} 2b_1 & 0 & b_2  \\
b_2 & 3b_1 & b_3 \\
0 & 0 & b_1  \\
\end{array}\right)
$$
for some $d_1$, $d_2$, $d_3$, $b_1$, $b_2$, $b_3 \in k$. Let now
$\alpha \in k$ and consider $g_0 = c_1 e_1 + c_2 e_2 + c_3 e_3$,
for some $c_1$, $c_2$, $c_3 \in k$. We can easily see that the
$5$-tuple $(g_0, \, \alpha, \, \lambda, \, D, \, \Delta)$
satisfies the compatibilities (F1) - (F7), i.e. it is a flag datum
of the first kind if and only if coincides with one of the three
flag datums described in $(1)$. For instance, the compatibility
(F1) is fulfilled if and only if $\lambda_0 (\alpha + d_3) =
\lambda_0 \, b_1 = 0$. This last equality leads us to consider two
cases, namely $\lambda_0 = 0$ or $\lambda_0 \neq 0$. It is now
straightforward to describe ${\mathcal F}_1 \, (\mathfrak{g})$
(without depending on the characteristic of $k$).

Analogously, we can describe ${\mathcal F}_2 \, (\mathfrak{g})$. A
non-trivial map $\nu: \mathfrak{g} \to k$ satisfies the first
compatibility of (G1) if and only if $\nu$ is given by $\nu (e_1)
= \nu (e_2) = 0$ and $\nu (e_3) = \nu_0$, for some $\nu_0 \in
k^*$. For such a map $\nu$, we can easily show that $(D, \Delta)$
satisfies the compatibilities (G4) and (G5) if and only if $D$ and
$\Delta$ are given by \equref{flagnasol}. By considering again
$g_0 = c_1 e_1 + c_2 e_2 + c_3 e_3$, for some $c_1$, $c_2$, $c_3
\in k$ we see that the compatibility (G1) is fulfilled if and only
if $c_3 = 0$. The last compatibility of (G3) is equivalent to:
$$
2 \, c_1 = 2 \, \nu_0^{-2} \, d_3^2, \qquad c_1 + 2 \, \nu_0 c_2 =
2 \, \nu_0^{-1} d_1 d_3 -  \nu_0^{-1} d_3^2
$$
The above two compatibilities are the ones that lead us to the
description of ${\mathcal F}_2 \, (\mathfrak{g})$ depending on the
characteristic of $k$. Moreover, these computations provide also
the description of the classifying object ${\mathcal H} {\mathcal
L}^{2}_{\mathfrak{g}} (k, \mathfrak{g} )$. If  ${\rm char} (k)
\neq 2$ then
\begin{equation}\eqlabel{balamuc}
{\mathcal H} {\mathcal L}^{2}_{\mathfrak{g}} (k, \mathfrak{g} )
\cong (k^5/\equiv_{11}) \, \sqcup \, ((k^* \times k^4) /
\equiv_{12}) \, \sqcup \, ((k^* \times k^3) / \equiv_{13}) \,
\sqcup \, ((k^* \times k^3) / \equiv_{2})
\end{equation}
where $\equiv_{1i}$ are the equivalence relations
\equref{lzeci1}-\equref{lzeci4} while $\equiv_{2}$ is the
equivalence relation \equref{lzeci1a}-\equref{lzeci3a}. In the
case when ${\rm char} (k) = 2$, then the last term of
\equref{balamuc} is replaced by $(k^* \times k^4) / \equiv_{2}$.
\end{example}

\section{Special cases of unified products}\selabel{cazurispeciale}
In this section we deal with two special cases of the unified
product namely the crossed (resp. bicrossed) product of two
Leibniz algebras. We emphasize the problem for which each of these
products is responsible. We use the following convention: if one
of the maps $\triangleleft$, $\triangleright$, $\leftharpoonup$,
$\rightharpoonup$, $f$ or $\{-, \, -\}$ of an extending datum
$\Omega(\mathfrak{g}, V) = \bigl(\triangleleft, \, \triangleright,
\, \leftharpoonup, \, \rightharpoonup, \, f, \{-, \, -\} \bigl)$
is trivial then we will omit it from the $6$-tuple
$\bigl(\triangleleft, \, \triangleright, \, \leftharpoonup, \,
\rightharpoonup, \, f, \{-, \, -\} \bigl)$.

\subsection*{Crossed products and the extension problem for
Leibniz algebras} We shall highlight a first special case of the
unified product, namely the crossed product of Leibniz algebras
that is the key player in the study of the extension problem in
its full generality. Let $\Omega (\mathfrak{g}, V) = \bigl(
\triangleleft, \, \triangleright, \, \leftharpoonup, \,
\rightharpoonup, \, f, \{-, \, -\} \bigl)$ be an extending datum
of a Leibniz algebra $\mathfrak{g}$ through $V$ such that
$\triangleleft$ and $\rightharpoonup$ are both trivial, i.e. $x
\triangleleft g = g \rightharpoonup x = 0$, for all $x \in V$ and
$g \in \mathfrak{g}$. Then, it follows from \thref{1} that
$\Omega(\mathfrak{g}, V) = \bigl(\triangleright, \,
\leftharpoonup, \, f, \{-, \, -\} \bigl) $ is a Leibniz extending
structure of $\mathfrak{g}$ through $V$ if and only if $(V, \{-,
-\})$ is a Leibniz algebra and $(\mathfrak{g}, \, V, \,
\triangleright, \leftharpoonup, \, f)$ is a \emph{crossed system}
of Leibniz algebras, i.e. the following compatibilities hold for
any $g$, $h\in \mathfrak{g}$ and $x$, $y$, $z\in V$:
\begin{enumerate}
\item[(CS1)] $[g, \, h] \leftharpoonup x = [g, \, h \leftharpoonup
x] + [g \leftharpoonup x, \, h]$; \item[(CS2)] $g \leftharpoonup
\{x, \, y\} = (g \leftharpoonup x) \leftharpoonup y - (g
\leftharpoonup y) \leftharpoonup x - [g, \, f(x, y)]$;
\item[(CS3)] $x \rhd f(y, \, z) = f(x, \, y) \leftharpoonup z -
f(x, \, z) \leftharpoonup y + f(\{x, \, y\}, \, z) - f(\{x, \,
z\}, \, y) - f(x, \, \{y, \, z\})$; \item[(CS4)] $x \rhd [g, \, h]
= [x \rhd g, \, h] - [x \rhd h, \, g]$; \item[(CS5)] $\{x, \, y\}
\rhd g = x \rhd (y \rhd g) + (x \rhd g) \leftharpoonup y - [f(x,
\, y), \, g]$; \item[(CS6)] $[g, \, h \leftharpoonup x] + [g, \, x
\rhd h] = 0$; \item[(CS7)] $x \rhd (y \rhd g) + x \rhd (g
\leftharpoonup y) = 0$.
\end{enumerate}
In this case, the associated unified product $\mathfrak{g}
\ltimes_{\Omega(\mathfrak{g}, V)} V$ will be denoted by
$\mathfrak{g} \#_{\triangleright, \leftharpoonup}^f \, V$ and we
shall call it the \emph{crossed product} of the Leibniz algebras
$\mathfrak{g}$ and $V$. Hence, the crossed product associated to
the crossed system $(\mathfrak{g}, V, \triangleright,
\leftharpoonup, f)$ is the Leibniz algebra defined as follows:
$\mathfrak{g} \#_{\triangleright, \leftharpoonup}^f \, V =
\mathfrak{g} \times \, V $ with the bracket given for any $g$, $h
\in \mathfrak{g}$ and $x$, $y \in V$ by:

\begin{equation}\eqlabel{brackcrosspr}
[(g, x), \, (h, y)] := \bigl( [g, \, h] + x \triangleright h + g
\leftharpoonup y + f(x, y), \, \{x, \, y \} \bigl)
\end{equation}

The crossed product of Leibniz algebras is the object responsible
for answering the following special case of the ES problem, which
is a generalization of the extension problem: \emph{Let
$\mathfrak{g}$ be a Leibniz algebra, $E$ a vector space containing
$\mathfrak{g}$ as a subspace. Describe and classify all Leibniz
algebra structures on $E$ such that $\mathfrak{g}$ is a two-sided
ideal of $E$.} The classical extension problem initiated in
\cite{LoP} is a special case of this question if we require the
additional assumption on the quotient $E/ \mathfrak{g}$ to be
isomorphic to a given Leibniz algebra $\mathfrak{h}$.

Indeed, let $(\mathfrak{g}, V, \triangleright, \leftharpoonup, f)$
be a crossed system of two Leibniz algebras. Then, $\mathfrak{g}
\cong \mathfrak{g} \times \{0\}$ is a two-sided ideal in the
crossed product $\mathfrak{g} \#_{\triangleright,
\leftharpoonup}^f \, V$ since $[(g, 0), \, (h, y)] := \bigl( [g,
\, h] + g \leftharpoonup y , \, 0 \bigl)$ and $[(g, x), \, (h, 0)]
:= \bigl( [g, \, h] + x \triangleright h , \, 0 \bigl)$.
Conversely, we have:

\begin{corollary}\colabel{croslieide}
Let $\mathfrak{g}$ be a Leibniz algebra, $E$ a vector space
containing $\mathfrak{g}$ as a subspace. Then any Leibniz algebra
structure on $E$ that contains $\mathfrak{g}$ as a two-sided ideal
is isomorphic to a crossed product of Leibniz algebras
$\mathfrak{g} \#_{\triangleright, \leftharpoonup}^f \, V$ and the
isomorphism can be chosen to stabilize $\mathfrak{g}$ and
co-stabilize $V$.
\end{corollary}

\begin{proof}
Let $[-,\, -]$ be a Leibniz algebra structure on $E$ such that
$\mathfrak{g}$ is a two-sided ideal in $E$. In particular,
$\mathfrak{g}$ is a subalgebra of $E$ and hence we can apply
\thref{classif}. In this case the actions $\triangleleft =
\triangleleft_p$ and $\rightharpoonup = \rightharpoonup_{p}$ of
the Leibniz extending structure $\Omega(\mathfrak{g}, V) =
\bigl(\triangleleft_p, \, \triangleright_p, \, f_p, \{-, \, -\}_p
\bigl)$ constructed in the proof of \thref{classif} are both
trivial since for any $x \in V$ and $g \in \mathfrak{g}$ we have
that $[x, g]$, $[g, x] \in \mathfrak{g}$ and hence $p ([x, g]) =
[x, g]$ and $p([g, x]) = [g, x]$. Thus, $x \triangleleft_p g = g
\rightharpoonup_{p} x = 0$ and hence the Leibniz extending
structure $\Omega(\mathfrak{g}, V) = \bigl(\triangleleft,
\triangleright, \leftharpoonup, \rightharpoonup,  f, \{-, \,
-\}\bigl)$ constructed in the proof of \thref{classif} is
precisely a crossed system of Leibniz algebras and the unified
product $\mathfrak{g} \ltimes_{\Omega(\mathfrak{g}, V)} V =
\mathfrak{g} \#_{\triangleright, \leftharpoonup}^f \, V $ is the
crossed product of $\mathfrak{g}$ and $V = {\rm Ker}(p)$.
\end{proof}

Let $\mathfrak{g}$ and $\mathfrak{h}$ be two given Leibniz
algebras. The extension problem asks for the classification of all
extensions of $\mathfrak{h}$ by $\mathfrak{g}$, i.e. of all
Leibniz algebras $\mathfrak{E}$ that fit into an exact sequence
\begin{eqnarray} \eqlabel{extencros0}
\xymatrix{ 0 \ar[r] & \mathfrak{g} \ar[r]^{i} & \mathfrak{E}
\ar[r]^{\pi} & \mathfrak{h} \ar[r] & 0 }
\end{eqnarray}
The classification is up to an isomorphism of Leibniz algebras
that stabilizes $\mathfrak{g}$ and co-stabilizes $\mathfrak{h}$
and we denote by ${\mathcal E} {\mathcal P} (\mathfrak{h}, \,
\mathfrak{g})$ the isomorphism classes of all extensions of
$\mathfrak{h}$ by $\mathfrak{g}$ up to this equivalence relation.
If $\mathfrak{g}$ is abelian, then ${\mathcal E} {\mathcal P}
(\mathfrak{h}, \, \mathfrak{g}) \cong {\rm HL}^2 (\mathfrak{h}, \,
\mathfrak{g})$, where ${\rm HL}^2 (\mathfrak{h}, \, \mathfrak{g})$
is the the second cohomology group \cite[Proposition 1.9]{LoP}.
The crossed product is the tool to approach the extension problem
in its full generality, leaving aside the abelian case. Let us
explain this briefly. Consider $\mathfrak{g}$ and $\mathfrak{h}$
be two Leibniz algebras and we denote by ${\mathcal C}{\mathcal S}
\, (\mathfrak{h}, \, \mathfrak{g} )$ the set of all triples
$(\triangleright, \, \leftharpoonup, \, f)$ such that
$(\mathfrak{g}, \, \mathfrak{h}, \, \triangleright, \,
\leftharpoonup, \, f)$ is a crossed system of Leibniz algebras.
First we remark that, if $(\mathfrak{g}, \, \mathfrak{h}, \,
\triangleright, \, \leftharpoonup, \, f)$ is a crossed system,
then the crossed product $\mathfrak{g} \#_{\triangleright,
\leftharpoonup}^f \, \mathfrak{h}$ is an extension of
$\mathfrak{h}$ by $\mathfrak{g}$ via
\begin{eqnarray} \eqlabel{extencros001}
\xymatrix{ 0 \ar[r] & \mathfrak{g} \ar[r]^{i_{\mathfrak{g}}} &
 \mathfrak{g}
\#_{\triangleright, \leftharpoonup}^f \, \mathfrak{h}
\ar[r]^{\pi_{\mathfrak{h}}} & \mathfrak{h} \ar[r] & 0 }
\end{eqnarray}
where $i_{\mathfrak{g}} (g) = (g, 0)$ and $\pi_{\mathfrak{h}} (g,
h) = h$ are the canonical maps. Conversely, \coref{croslieide}
shows that any extension $\mathfrak{E}$ of $\mathfrak{h}$ by
$\mathfrak{g}$ is equivalent to a crossed product extension of the
form \equref{extencros001}. Thus, the classification of all
extensions of $\mathfrak{h}$ by $\mathfrak{g}$ reduces to the
classification of all crossed products $\mathfrak{g}
\#_{\triangleright, \leftharpoonup}^f \, \mathfrak{h}$ associated
to all crossed systems of Leibniz algebras $(\mathfrak{g},
\mathfrak{h}, \triangleright, \leftharpoonup, f)$.
\deref{echiaab}, in the special case of crossed systems, takes the
following simplified form: two triples $(\triangleright,\,
\leftharpoonup, \, f)$ and $(\triangleright ',\, \leftharpoonup ',
\, f')$ of ${\mathcal C}{\mathcal S} \, (\mathfrak{h}, \,
\mathfrak{g} )$ are \emph{cohomologous} and we denote this by
$(\triangleright,\, \leftharpoonup, \, f) \approx (\triangleright
',\, \leftharpoonup ', \, f')$ if there exists a linear map $r:
\mathfrak{h} \to \mathfrak{g}$ such that:
\begin{eqnarray*}
x \triangleright  g &=& x \triangleright ' g  + [r(x), \, g]\\
g \leftharpoonup x &=& g \leftharpoonup ' x + [g, \, r(x)]\\
f(x, y) &=& f '(x, y) + [r(x), \, r(y)] - r \bigl(\{x, \,
y\}\bigl) + \, r (x) \leftharpoonup ' y + x \triangleright ' r (y)
\end{eqnarray*}
for all $g \in \mathfrak{g}$, $x$, $y \in \mathfrak{h}$. Then, as
we mentioned before \deref{echiaab}, $(\triangleright,\,
\leftharpoonup, \, f) \approx (\triangleright ',\, \leftharpoonup
', \, f')$ if and only if there exists $\psi : \mathfrak{g}
\#_{\triangleright, \leftharpoonup}^f \, \mathfrak{h} \to
\mathfrak{g} \#_{\triangleright ', \leftharpoonup '}^{f'} \,
\mathfrak{h}$ an isomorphism of Leibniz algebras that stabilizes
$\mathfrak{g}$ and co-stabilizes $\mathfrak{h}$. As a special case
of \thref{main1}, we obtain the theoretical answer to the
extension problem in the general (non-abelian) case:

\begin{corollary}\colabel{extprobra}
Let $\mathfrak{g}$ and $\mathfrak{h}$ be two arbitrary Leibniz
algebras. Then  $\approx$ is an equivalence relation on the set
${\mathcal C}{\mathcal S} \, (\mathfrak{h}, \, \mathfrak{g} )$ of
all crossed systems and the map
$$
{\mathbb H} {\mathbb L}^2 (\mathfrak{h}, \, \mathfrak{g}) :=
{\mathcal C}{\mathcal S} \, (\mathfrak{h}, \, \mathfrak{g} )/
\approx \,\, \longrightarrow {\mathcal E} {\mathcal P}
(\mathfrak{h}, \, \mathfrak{g}), \qquad
\overline{(\triangleright,\, \leftharpoonup, \, f)} \mapsto
\mathfrak{g} \#_{\triangleright, \leftharpoonup}^f \, \mathfrak{h}
$$
is a bijection between sets, where $\overline{(\triangleright,\,
\leftharpoonup, \, f)}$ is the equivalence class of
$(\triangleright,\, \leftharpoonup, \, f)$ via $\approx$.
\end{corollary}

If $\mathfrak{g}$ is an abelian Leibniz algebra, then ${\mathbb H}
{\mathbb L}^2 (\mathfrak{h}, \, \mathfrak{g})$ coincides with the
second cohomology group ${\rm HL}^2 (\mathfrak{h}, \,
\mathfrak{g})$ constructed in \cite{LoP}. The explicit answer to
the extension problem for two given Leibniz algebras
$\mathfrak{g}$ and $\mathfrak{h}$ will be given once we compute
the non-abelian cohomological object ${\mathbb H} {\mathbb L}^2
(\mathfrak{h}, \, \mathfrak{g})$ which in general is a highly
non-trivial problem. A detailed study of this object for various
Leibniz algebras will be given elsewhere; here we give only one
example that corresponds to the case when $\mathfrak{h} := k$, the
abelian Leibniz algebra of dimension $1$, as this is a special
case of \thref{clasdim1} and \reref{alalatfl}.

\begin{corollary} \colabel{noucalculex}
Let $\mathfrak{g}$ be a Leibniz algebra with $\{e_i \, | \, i\in I
\}$ as a basis. Then
$$
{\mathbb H} {\mathbb L}^2 (k, \, \mathfrak{g}) \cong {\mathcal D}
(g)/ \approx
$$
where ${\mathcal D} (g)$ is the space of all pointed double
derivations of $\mathfrak{g}$ and $\approx$ is the equivalence
relation defined by \equref{lzeci1per2}-\equref{lzeci4per2}. In
particular, any extension of $k$ by $\mathfrak{g}$ is isomorphic
to the Leibniz algebra having $\{x, \, e_i \, | \, i\in I \}$ as a
basis and the bracket $\left[-, \, - \right]_{(g_0, \, D, \,
\Delta)}$ defined for any $i\in I$ by:
\begin{eqnarray*}
\left[e_i, \, e_j \right]_{(g_0, \, D, \, \Delta)} &:=& \left[e_i,
e_j \right], \quad \left[x, \, x \right]_{(g_0, \, D, \, \Delta)}
:= g_0 \\
\left[e_i, \, x \right]_{(g_0, \, D, \, \Delta)} &:=& \Delta(e_i),
\quad \,\, \left[x, \, e_i \right]_{(g_0, \, D, \, \Delta)} :=
D(e_i)
\end{eqnarray*}
for some $(g_{0}, \, D, \, \Delta) \in {\mathcal D} (g)$.
\end{corollary}

\begin{proof}
Follows from \thref{clasdim1} since the set of crossed systems
${\mathcal C}{\mathcal S} \, (k, \, \mathfrak{g} )$ is precisely
the set $\Omega (\mathfrak{g}, k) = \bigl( \triangleleft, \,
\triangleright, \, \leftharpoonup, \, \rightharpoonup, \, f, \{-,
\, -\} \bigl)$ of all Leibniz extending structures of
$\mathfrak{g}$ through $k$ having the actions $\triangleleft$ and
$\rightharpoonup$ both trivial. Moreover, any extension $E$ of $k$
by $\mathfrak{g}$ is a Leibniz algebra containing $\mathfrak{g}$
as a subalgebra of codimension $1$. In this context, the
compatibility conditions (F1)-(F7) that define a flag datum of the
first kind collapses to \equref{flagper1}-\equref{flagper4}. The
fact that $\triangleleft$ is the trivial action implies that
$\lambda = 0$. The Leibniz algebra from the statement is the
unified (crossed) product defined by
\equref{primuunif1}-\equref{primuunif2}.
\end{proof}

In the next example we compute explicitly the object ${\mathbb H}
{\mathbb L}^2 (k, \, \mathfrak{g})$ for a certain Leibniz algebra
$\mathfrak{g}$.

\begin{example} \exlabel{calexpext}
Let $\mathfrak{g}$ be the $3$-dimensional Leibniz algebra with the
basis $\{e_{1}, e_{2}, e_{3}\}$ and the bracket defined by: $
[e_{1}, \, e_{3}] = e_{2}$, $[e_{3}, \, e_{3}] = e_{1}$. A little
computation, similar to one performed in \exref{dim34}, shows that
the set ${\mathcal D} (g)$ identifies with the set of all
$6$-tuples $(c, \, b_1, \, b_2, \, b_3, \, d_1, \, d_2) \in k^6$
which satisfy:
$$
b_1 \, (d_1 - b_2) = 0
$$
The bijection is defined such that $(g_{0}, \, D, \, \Delta) \in
{\mathcal D} (g)$ corresponding to $(c, \, b_1, \, b_2, \, b_3, \,
d_1, \, d_2)$ is given by
$$
g_0 := (2\, b_1 d_2 + b_2 d_1 - b_1 b_3) \, e_1 + c \, e_2, \quad
D :=
\left( \begin{array}{ccc} 0 & 0 & d_1  \\
0 & 0 & d_2 \\
0 & 0 & -b_1  \\
\end{array}\right)
\quad \Delta :=
\left( \begin{array}{ccc} 2b_1 & 0 & b_2  \\
b_2 & 3b_1 & b_3 \\
0 & 0 & b_1  \\
\end{array}\right)
$$
The compatibility condition $b_1 \, (d_1 - b_2) = 0$ imposes a
discussion on whether $b_1 = 0$ or $b_1 \neq 0$. This leads to the
description of ${\mathbb L}^2 (k, \, \mathfrak{g})$ as the
following coproduct of sets:
$$
{\mathbb H} {\mathbb L}^2 (k, \, \mathfrak{g}) \cong (k^5 /
\approx_1  ) \sqcup (k^* \times k^4/ \approx_2 ), \qquad {\rm
where:}
$$
$\approx_1$ is the following relation on $k^5$: $(c, \, b_2, \,
b_3, \, d_1, \, d_2) \approx_1 (c', \, b'_2, \, b'_3, \, d'_1, \,
d'_2)$ if and only if there exist $u$, $v\in k$ such that
$$
c = c' + uv, \quad b_2 = b'_2 + v, \quad b_3 = b'_3 + v, \quad d_1
= d'_1 + v, \quad d_2 = d'_2 + u
$$
and $\approx_2$ is the relation on $k^* \times k^4$ defined by:
$(b_1, \, c, \, b_3, \, d_1, \, d_2) \approx_2 (b'_1, \, c', \,
b'_3, \, d'_1, \, d'_2)$ if and only if $b_1 = b'_1$, $b_3 =
b'_3$, $d_2 = d'_2$ and there exist $v$, $w \in k$ such that
$$
c = c' + v d'_1 + 3 w b'_1 + (d_1 - d'_1) \, (v + d'_2 + b'_3)
$$
\end{example}

\subsection*{Bicrossed products and the factorization problem for
Leibniz algebras} The concept of a matched pair of Lie algebras
was introduced in \cite[Theorem 4.1]{majid} and independently in
\cite[Theorem 3.9]{LW}. For any such matched pair of Lie algebras
a new Lie algebra, called the \emph{bicrossed product} is
constructed under the name of bicrossproduct in \cite[Theorem
4.1]{majid}, double cross sum in \cite[Proposition 8.3.2]{majid2},
double Lie algebra \cite[Definition 3.3]{LW}. Now we shall
introduce the concept of a matched pair of Leibniz algebras. As we
will see, in this case the definition is a lot more laborious.

\begin{definition} \delabel{mpLei}
A \emph{matched pair} of Leibniz algebras is a system
$(\mathfrak{g}, \, \mathfrak{h}, \, \triangleleft, \,
\triangleright, \, \leftharpoonup, \, \rightharpoonup)$ consisting
of two Leibniz algebras $(\mathfrak{g}, \, [-, -])$,
$(\mathfrak{h}, \, \{-, - \})$ and four bilinear maps
$\triangleleft : \mathfrak{h} \times \mathfrak{g} \to
\mathfrak{h}$, $\triangleright : \mathfrak{h} \times \mathfrak{g}
\to \mathfrak{g}$, $\leftharpoonup : \mathfrak{g} \times
\mathfrak{h} \to \mathfrak{g}$, $\rightharpoonup : \mathfrak{g}
\times \mathfrak{h} \to \mathfrak{h}$ satisfying the following
compatibilities for any $g$, $h \in \mathfrak{g}$, $x$, $y \in
\mathfrak{h}$:
\begin{enumerate}
\item[(MP1)] $(\mathfrak{h}, \, \lhd)$ is a right
$\mathfrak{g}$-module, i.e. $x \lhd [g, \, h] = (x \lhd g) \lhd h
- (x \lhd h) \lhd g$;

\item[(MP2)] $(\mathfrak{g}, \, \leftharpoonup)$ is a right
$\mathfrak{h}$-module, i.e. $g \leftharpoonup \{x, \, y\} = (g
\leftharpoonup x) \leftharpoonup y - (g \leftharpoonup y)
\leftharpoonup x$;

\item[(MP3)] $x \rhd [g, \, h] = [x \rhd g, \, h] - [x \rhd h,\,
g] + (x \lhd g) \rhd h - (x \lhd h) \rhd g$;

\item[(MP4)] $\{x, \, y\} \lhd g = x \lhd (y \rhd g) + (x \rhd g)
\rightharpoonup y + \{x, \, y \lhd g\} + \{x \lhd g, \, y\}$;

\item[(MP5)] $\{x, \, y\} \rhd g = x \rhd (y \rhd g) + (x \rhd g)
\leftharpoonup y$;

\item[(MP6)] $[g, \, h] \leftharpoonup x = [g, \, h \leftharpoonup
x] + [g \leftharpoonup x, \, h] + g \leftharpoonup (h
\rightharpoonup x) + (g \rightharpoonup x) \rhd h$;

\item[(MP7)] $[g, \, h] \rightharpoonup x = g \rightharpoonup (h
\rightharpoonup x) + (g \rightharpoonup x) \lhd h$;

\item[(MP8)] $g \rightharpoonup \{x,\, y \} = (g \leftharpoonup x)
\rightharpoonup y - (g \leftharpoonup y) \rightharpoonup x + \{g
\rightharpoonup x, \, y\} - \{g \rightharpoonup y, \, x\}$;

\item[(MP9)] $[g, \, h \leftharpoonup x] + [g, \, x \rhd h] + g
\leftharpoonup (h \rightharpoonup x) + g \leftharpoonup (x \lhd h)
= 0$;

\item[(MP10)] $x \rhd (y \rhd g) + x \rhd (g \leftharpoonup y) =
0$;

\item[(MP11)] $x \lhd (y \rhd g) + x \lhd (g \leftharpoonup y) +
\{x, \, y \lhd g\} + \{x, \, g \rightharpoonup y\} = 0$;

\item[(MP12)] $g \rightharpoonup (h \rightharpoonup x) + g
\rightharpoonup (x \lhd h) = 0$.
\end{enumerate}
\end{definition}

Let $(\mathfrak{g}, \, \mathfrak{h}, \, \triangleleft, \,
\triangleright, \, \leftharpoonup, \, \rightharpoonup)$ be a
matched pair of Leibniz algebras. Then $\mathfrak{g} \, \bowtie
\mathfrak{h} := \mathfrak{g} \, \times \mathfrak{h}$, as a vector
space, with the bracket defined for any $g$, $h \in \mathfrak{g}$
and $x$, $y \in \mathfrak{h}$ by
\begin{equation}\eqlabel{brackunifa}
[(g, x), \, (h, y)] := \bigl( [g, \, h] + x \triangleright h + g
\leftharpoonup y , \,\, \{x, \, y \} + x\triangleleft h + g
\rightharpoonup y \bigl)
\end{equation}
is a Leibniz algebra called the \emph{bicrossed product}
associated to the matched pair of Leibniz algebras $(
\mathfrak{g}, \, \mathfrak{h}, \, \triangleleft, \,
\triangleright, \, \leftharpoonup, \, \rightharpoonup )$. This
fact can be proved directly, but it can also be derived as a
special case of \thref{1}. Indeed, let $\mathfrak{g}$ be a Leibniz
algebra and $\Omega(\mathfrak{g}, V) = \bigl(\triangleleft, \,
\triangleright, \, \leftharpoonup, \, \rightharpoonup, \, f, \{-,
\, -\} \bigl)$ an extending datum of $\mathfrak{g}$ through $V$
such that $f$ is the trivial map, i.e. $f(x, y) = 0$, for all $x$,
$y\in V$. Then, we can easily see that $\Omega(\mathfrak{g}, V) =
\bigl(\triangleleft, \, \triangleright, \, \leftharpoonup, \,
\rightharpoonup, \, \{-, \, -\} \bigl)$ is a Leibniz extending
structure of $\mathfrak{g}$ through $V$ if and only if $(V, \,
\{-, -\})$ is a Leibniz algebra and $( \mathfrak{g}, V,
\triangleleft, \, \triangleright, \, \leftharpoonup, \,
\rightharpoonup)$ is a matched pair of Leibniz algebras in the
sense of \deref{mpLei}. In this case, the associated unified
product $\mathfrak{g} \ltimes_{\Omega(\mathfrak{g}, V)} V =
\mathfrak{g} \bowtie V$ is the bicrossed product of the matched
pair $( \mathfrak{g}, V, \triangleleft, \, \triangleright, \,
\leftharpoonup, \, \rightharpoonup)$ as defined by
\equref{brackunifa}.

The bicrossed product of two Leibniz algebras is the construction
responsible for the \emph{factorization problem}, which is a
special case of the ES problem and can be stated as follows:
\emph{Let $\mathfrak{g}$ and $\mathfrak{h}$ be two given Leibniz
algebras. Describe and classify all Leibniz algebras $\Xi$ that
factorize through $\mathfrak{g}$ and $\mathfrak{h}$, i.e. $\Xi$
contains $\mathfrak{g}$ and $\mathfrak{h}$ as Leibniz subalgebras
such that $\Xi = \mathfrak{g} + \mathfrak{h}$ and $\mathfrak{g}
\cap \mathfrak{h} = \{0\}$.} Indeed, using \thref{classif} we can
prove the following:

\begin{corollary}\colabel{bicrfactor}
A Leibniz algebra $\Xi$ factorizes through $\mathfrak{g}$ and
$\mathfrak{h}$ if and only if there exists a matched pair of
Leibniz algebras $(\mathfrak{g}, \mathfrak{h}, \, \triangleleft,
\, \triangleright, \, \leftharpoonup, \, \rightharpoonup)$ such
that $ \Xi \cong \mathfrak{g} \bowtie \mathfrak{h}$.
\end{corollary}

\begin{proof}
To start with, notice that any bicrossed product $\mathfrak{g}
\bowtie \mathfrak{h}$ factorizes through $\mathfrak{g} \cong
\mathfrak{g} \times \{0\}$ and $\mathfrak{h} \cong \{0\}\times
\mathfrak{h}$. Conversely, assume that $\Xi$ factorizes through
$\mathfrak{g}$ and $\mathfrak{h}$. Let $p: \Xi \to \mathfrak{g}$
be the $k$-linear projection of $\Xi$ on $\mathfrak{g}$, i.e. $p
(g + x) := g$, for all $g\in \mathfrak{g}$ and $x \in
\mathfrak{h}$. Now, we apply \thref{classif} for $V: = {\rm
Ker}(p) = \mathfrak{h}$. Since $V$ is a Leibniz subalgebra of $E
:= \Xi$, the map $f = f_p$ constructed in the proof of
\thref{classif} is the trivial map as $[x, y] \in V = {\rm Ker}
(p)$. Thus, the Leibniz extending structure $\Omega(\mathfrak{g},
V) = \bigl(\triangleleft, \, \triangleright, \, \leftharpoonup, \,
\rightharpoonup, \, f, \, \{-, \, -\}\bigl)$ constructed in the
proof of \thref{classif} is precisely a matched pair of Leibniz
algebra and the unified product $\mathfrak{g}
\ltimes_{\Omega(\mathfrak{g}, V)} V = \mathfrak{g} \bowtie V$ is
the bicrossed product of the matched pair $( \mathfrak{g}, V,
\triangleleft, \, \triangleright, \, \leftharpoonup, \,
\rightharpoonup)$. Explicitly, the matched pair $(\mathfrak{g},
\mathfrak{h}, \, \triangleleft = \triangleleft_p, \,
\triangleright = \triangleright_p, \, \leftharpoonup =
\leftharpoonup_{p}, \, \rightharpoonup = \rightharpoonup_{p})$ is
given by:
\begin{eqnarray}
x \triangleright g := p \bigl(\left[x, \, g \right]\bigl), \qquad
x \triangleleft g : = \left[x, \, g \right] - p \bigl(\left[x, \,
g \right]\bigl) \eqlabel{mpcanonic1} \\ g \leftharpoonup x :=
p\bigl([g, \, x]\bigl), \qquad g \rightharpoonup x := [g, \, x] -
p\bigl([g, \, x]\bigl)  \eqlabel{mpcanonic2}
\end{eqnarray}
for all $x \in \mathfrak{h}$ and $g \in \mathfrak{g}$.
\end{proof}

From now on the matched pair constructed in \equref{mpcanonic1}
and \equref{mpcanonic2} will be called the \emph{canonical matched
pair} associated to the factorization $\Xi = \mathfrak{g} +
\mathfrak{h}$ of $\Xi$ through $\mathfrak{g}$ and $\mathfrak{h}$.
Based on \coref{bicrfactor} the factorization problem can be
restated in a computational manner as follows: Let $\mathfrak{g}$
and $\mathfrak{h}$ be two given Leibniz algebras. Describe the set
of all matched pairs $(\mathfrak{g}, \mathfrak{h}, \,
\triangleleft, \, \triangleright, \, \leftharpoonup, \,
\rightharpoonup)$  and classify up to an isomorphism all bicrossed
products $\mathfrak{g} \bowtie \mathfrak{h}$. A detailed study of
this problem will be given somewhere else.

\begin{example}
Let $\mathfrak{g}$ be the $3$-dimensional Leibniz algebra
considered in \exref{dim34} and $k$ be the $1$-dimensional
(abelian) Leibniz algebra. Then all bicrossed products
$\mathfrak{g} \bowtie k$ can be explicitly described as a special
case of \exref{dim34}. To this end we need to consider $g_0 := 0$
and $\alpha := 0$ in the unified products associated to all flag
datums of the first kind provided in \exref{dim34} and $g_0 := 0$
in the unified products associated to all flag datums of the
second kind. For instance, by taking $d_3 = 0$ in the Leibniz
algebra $\mathfrak{g}_{21}^{(\nu_0, \, d_1, \, d_2, , d_3)}$ of
\exref{dim34} we obtain the bicrossed product $\mathfrak{g}
\bowtie k$ which is a $4$-dimensional Leibniz algebra with the
basis $\{e_1, e_2, e_3, x\}$ and the bracket given by:
\begin{eqnarray*}
\left[e_1, \, e_3\right] &=& e_2, \quad \left[e_3, \, e_3\right] =
e_1, \quad \left[e_3, \, x \right] = - d_1 e_1 - (d_2 - \nu_0^{-1}
d_1) e_2 + \nu_0 x, \\
\quad \left[x, \, e_3\right] &=& d_1 \, e_1 + d_2 \, e_2 - \nu_0
\, x
\end{eqnarray*}
for all $\nu_0 \in k^*$ and $d_1$, $d_2 \in k$. This Leibniz
algebra is the bicrossed product associated to the following
matched pair $(\mathfrak{g}, k, \, \triangleleft, \,
\triangleright, \, \leftharpoonup, \, \rightharpoonup)$:
\begin{eqnarray*}
x \triangleleft e_3 &:=& - \nu_0 \, x, \quad x \triangleright e_3
:= d_1 e_1 + d_2 e_2, \\
e_3 \rightharpoonup x &:=& \nu_0 x, \quad \,\, e_3 \leftharpoonup
x := -d_1 e_1 + ( - d_2 + \nu_0^{-1} d_1) e_2
\end{eqnarray*}
where the undefined actions are zero and $x$ is a basis of $k$.
Another example of a matched pair of Leibniz algebras and the
corresponding bicrossed product will be given in \exref{mpleib}.
\end{example}

\section{Classifying complements for extensions of
Leibniz algebras} \selabel{complements}

This section is devoted to the classifying complements (CC)
problem whose statement was given in the Introduction. Let
$\mathfrak{g} \subseteq \Xi$ be a Leibniz subalgebra of $\Xi$. A
Leibniz subalgebra $\mathfrak{h}$ of $\Xi$ is called a
\emph{complement} of $\mathfrak{g}$ in $\Xi$ (or a
\emph{$\mathfrak{g}$-complement} of $\Xi$) if $\Xi = \mathfrak{g}
+ \mathfrak{h}$ and $\mathfrak{g} \cap \mathfrak{h} = \{0\}$. If
$\mathfrak{h}$ is a complement of $\mathfrak{g}$ in $\Xi$,
\coref{bicrfactor} shows that $ \Xi \cong \mathfrak{g} \bowtie
\mathfrak{h}$, where $\mathfrak{g} \bowtie \mathfrak{h}$ is the
bicrossed product associated to the canonical matched pair of the
factorization $\Xi = \mathfrak{g} + \mathfrak{h}$ as constructed
in \equref{mpcanonic1} and \equref{mpcanonic2}.

We denote by ${\mathcal F} (\mathfrak{g}, \, \Xi)$ the (possibly
empty) isomorphism classes of all $\mathfrak{g}$-complements of
$\Xi$. The \emph{factorization index} of $\mathfrak{g}$ in $\Xi$
is defined by $[\Xi : \mathfrak{g}]^f := |\, {\mathcal F}
(\mathfrak{g}, \, \Xi) \,|$ as a numerical measure of the (CC)
problem.

\begin{definition} \delabel{deformaplie}
Let $(\mathfrak{g}, \, \mathfrak{h}, \, \triangleright, \,
\triangleleft, \, \leftharpoonup, \, \rightharpoonup)$ be a
matched pair of Leibniz algebras. A linear map $r: \mathfrak{h}
\to \mathfrak{g}$ is called a \emph{deformation map} of the
matched pair $(\mathfrak{g}, \mathfrak{h}, \triangleright,
\triangleleft, \, \leftharpoonup, \, \rightharpoonup)$ if the
following compatibility holds for any $x$, $y \in \mathfrak{h}$:
\begin{equation}\eqlabel{factLie}
r\bigl([x, \,y]\bigl) \, - \, \bigl[r(x), \, r(y)\bigl] = x
\triangleright r(y) + r(x) \leftharpoonup y - r \bigl(x
\triangleleft r(y) + r(x) \rightharpoonup y \, \bigl)
\end{equation}
\end{definition}

We denote by ${\mathcal D}{\mathcal M} \, (\mathfrak{h},
\mathfrak{g} \, | \, (\triangleright, \triangleleft,
\leftharpoonup, \, \rightharpoonup) )$ the set of all deformation
maps of the matched pair $(\mathfrak{g}, \mathfrak{h},
\triangleright, \triangleleft, \leftharpoonup, \rightharpoonup)$.
The trivial map $r(x) = 0$, for all $x \in \mathfrak{h}$, is of
course a deformation map. The right hand side of \equref{factLie}
measures how far $r: \mathfrak{h} \to \mathfrak{g}$ is from being
a Leibniz algebra map. Using this concept which will play a key
role in solving the (CC) problem, we introduce the following
deformation of a Leibniz algebra:

\begin{theorem}\thlabel{deforLie}
Let $\mathfrak{g}$ be a Leibniz subalgebra of $\Xi$,
$\mathfrak{h}$ a given $\mathfrak{g}$-complement of $\Xi$ and $r:
\mathfrak{h} \to \mathfrak{g}$ a deformation map of the associated
canonical matched pair $(\mathfrak{g}, \mathfrak{h},
\triangleright, \triangleleft, \leftharpoonup, \rightharpoonup)$.

$(1)$ Let $f_{r}: \mathfrak{h} \to \Xi = \mathfrak{g} \bowtie
\mathfrak{h}$ be the $k$-linear map defined for any $x \in
\mathfrak{h}$ by:
$$f_{r}(x) = (r(x),\, x)$$
Then $\widetilde{\mathfrak{h}} : = {\rm Im}(f_{r})$ is a
$\mathfrak{g}$-complement of $\Xi$.

$(2)$ $\mathfrak{h}_{r} := \mathfrak{h}$, as a vector space, with
the new bracket defined for any $x$, $y \in \mathfrak{h}$ by:
\begin{equation}\eqlabel{rLiedef}
[x, \, y]_{r} := [x, \, y] + x \triangleleft r(y) + r(x)
\rightharpoonup y
\end{equation}
is a Leibniz algebra called the $r$-deformation of $\mathfrak{h}$.
Furthermore, $\mathfrak{h}_{r} \cong \widetilde{\mathfrak{h}}$, as
Leibniz algebras.
\end{theorem}

\begin{proof}
$(1)$ To start with, we will prove that $\widetilde{\mathfrak{h}}
= \{\bigl(r(x),\, x \bigl) ~|~ x \in \mathfrak{h}\}$ is a Leibniz
subalgebra of $\mathfrak{g} \bowtie \mathfrak{h} = \Xi $. Indeed,
for all $x$, $y \in \mathfrak{h}$ we have:
\begin{eqnarray*}
\bigl[ (r(x), x), (r(y), y) \bigl]
&\stackrel{\equref{brackunifa}}{=}& \Bigl(\underline{\bigl[r(x),
r(y)\bigl] + x \triangleright r(y) + r(x) \leftharpoonup y}, \,
[x, y] + x \triangleleft r(y) + r(x) \rightharpoonup y \Bigl)\\
&\stackrel{\equref{factLie}}{=}& \Bigl(r ([x, \, y] + x
\triangleleft r(y) + r(x) \rightharpoonup y), \, [x, \, y] + x
\triangleleft r(y) + r(x) \rightharpoonup y\Bigl)
\end{eqnarray*}
i.e. $\bigl[ (r(x), \, x), \, (r(y), \, y) \bigl] \in
\widetilde{\mathfrak{h}}$. Moreover, it is straightforward to see
that $\mathfrak{g} \, \cap \,  \widetilde{\mathfrak{h}} = \{0\}$
and $(g,\, x) = \bigl(g - r(x), \, 0\bigl) + \bigl(r(x),\, x
\bigl) \in \mathfrak{g} + \widetilde{\mathfrak{h}}$ for all $g \in
\mathfrak{g}$, $x \in \mathfrak{h}$. Here, we view $\mathfrak{g}
\cong \mathfrak{g} \times \{0\}$ as a subalgebra of $\mathfrak{g}
\bowtie \mathfrak{h}$. Therefore, $\widetilde{\mathfrak{h}}$ is a
$\mathfrak{g}$-complement of $\Xi = \mathfrak{g} \bowtie
\mathfrak{h}$.

$(2)$ We denote by $\widetilde{f_{r}} : \mathfrak{h} \to
\widetilde{\mathfrak{h}}$ the linear isomorphism induced by
$f_{r}$. We will prove that $\widetilde{f_{r}}$ is also a Leibniz
algebra map if we consider on $\mathfrak{h}$ the bracket given by
\equref{rLiedef}. Indeed, for any $x$, $y \in \mathfrak{h}$ we
have:
\begin{eqnarray*}
\widetilde{f_{r}}\bigl([x,\,
y]_{r}\bigl)&\stackrel{\equref{rLiedef}}{=}&
\widetilde{f_{r}}\bigl([x, \, y] + x \triangleleft r(y) + r(x) \rightharpoonup y\bigl)\\
&{=}& \Bigl(\underline{r \bigl([x, \, y] + x \triangleleft r(y) +
r(x) \rightharpoonup y\bigl)},\, [x, \, y] +
x\triangleleft r(y)+ r(x) \rightharpoonup y\Bigl)\\
&\stackrel{\equref{factLie}}{=}& \Bigl([r(x), \, r(y)] + x
\triangleright r(y) + r((x) \leftharpoonup y,\, [x, \, y] + x
\triangleleft r(y) + r(x) \rightharpoonup y\Bigl)\\
&\stackrel{\equref{brackunifa}}{=}& [(r(x),\, x), \, (r(y), \, y)]
= [\widetilde{f_{r}}(x), \, \widetilde{f_{r}}(y)]
\end{eqnarray*}
Therefore, $\mathfrak{h}_{r}$ is a Leibniz algebra and the proof
is now finished.
\end{proof}

The following is the converse of \thref{deforLie}: it proves that
all $\mathfrak{g}$-complements of $\Xi$ are $r$-deformations of a
given complement.

\begin{theorem} \thlabel{descrierecomlie}
Let $\mathfrak{g}$ be a Leibniz subalgebra of $\Xi$,
$\mathfrak{h}$ a given $\mathfrak{g}$-complement of $\Xi$ with the
associated canonical matched pair of Leibniz algebras
$(\mathfrak{g}, \mathfrak{h}, \triangleright, \triangleleft, ,
\leftharpoonup, \rightharpoonup)$. Then $\overline{\mathfrak{h}}$
is a $\mathfrak{g}$-complement of $\Xi$ if and only if there
exists an isomorphism of Leibniz algebras $\overline{\mathfrak{h}}
\cong \mathfrak{h}_{r}$, for some deformation map $r: \mathfrak{h}
\to \mathfrak{g}$ of the matched pair $(\mathfrak{g},
\mathfrak{h}, \triangleright, \triangleleft, \leftharpoonup,
\rightharpoonup)$.
\end{theorem}

\begin{proof}
Let $\overline{\mathfrak{h}}$ be an arbitrary
$\mathfrak{g}$-complement of $\Xi$. Since $\Xi = \mathfrak{g}
\oplus \mathfrak{h} = \mathfrak{g} \oplus \overline{\mathfrak{h}}$
we can find four $k$-linear maps:
$$
u: \mathfrak{h} \to \mathfrak{g}, \quad v: \mathfrak{h} \to
\overline{\mathfrak{h}}, \quad t:\overline{\mathfrak{h}} \to
\mathfrak{g}, \quad w: \overline{\mathfrak{h}} \to \mathfrak{h}
$$
such that for all $x \in \mathfrak{h}$ and $y \in
\overline{\mathfrak{h}}$ we have:
\begin{equation} \eqlabel{lie111}
x = u(x) \oplus v(x), \qquad y = t(y) \oplus w(y)
\end{equation}
By an easy computation it follows that $v: \mathfrak{h} \to
\overline{\mathfrak{h}}$ is a linear isomorphism of vector spaces.
We denote by $\tilde{v}: \mathfrak{h} \to \mathfrak{g} \bowtie
\mathfrak{h}$ the composition:
$$
\tilde{v} : \, \mathfrak{h} \, \stackrel{v} {\longrightarrow} \,
\overline{\mathfrak{h}} \, \stackrel{i}{\hookrightarrow} \, \Xi \,
= \,\mathfrak{g} \bowtie \mathfrak{h}
$$
Therefore, we have $\tilde{v}(x) \stackrel{\equref{lie111}}{=}
\bigl(-u(x),\, x\bigl)$, for all $x \in \mathfrak{h}$. Then we
shall prove that $r := - u$ is a deformation map and
$\overline{\mathfrak{h}} \cong \mathfrak{h}_{r}$. Indeed,
$\overline{\mathfrak{h}} = {\rm Im} (v) = {\rm Im} (\tilde{v})$ is
a Leibniz subalgebra of $\Xi = \mathfrak{g} \bowtie \mathfrak{h}$
and we have:
\begin{eqnarray*}
[\bigl(r(x),\, x\bigl), \, \bigl(r(y),\, y\bigl)]
&\stackrel{\equref{brackunifa}}{=}& \bigl([r(x), \, r(y)] + x
\triangleright r(y) +  r(x) \leftharpoonup y , \,[x, \, y] + \\
&&x \triangleleft r(y) + r(x) \rightharpoonup y \bigl) \, =
\bigl(r(z),\, z\bigl)
\end{eqnarray*}
for some $z \in \mathfrak{h}$. Thus, we obtain:
\begin{equation}\eqlabel{lie113}
r(z) = [r(x), \, r(y)] + x \triangleright r(y) +  r(x)
\leftharpoonup y, \qquad z = [x, \, y] + x \triangleleft r(y) +
r(x) \rightharpoonup y
\end{equation}

By applying $r$ to the second part of \equref{lie113} it follows
that $r$ is a deformation map of the matched pair $(\mathfrak{g},
\mathfrak{h}, \triangleright, \triangleleft, \leftharpoonup,
\rightharpoonup)$. Furthermore, \equref{lie113} and
\equref{rLiedef} show that $v: \mathfrak{h}_{r} \to
\overline{\mathfrak{h}}$ is also a Leibniz algebra map which
finishes the proof.
\end{proof}

In order to provide the classification of all complements we
introduce the following:

\begin{definition}\delabel{equivLie}
Let $(\mathfrak{g}, \mathfrak{h}, \triangleright, \triangleleft,
\leftharpoonup, \rightharpoonup)$ be a matched pair of Leibniz
algebras. Two deformation maps $r$, $R: \mathfrak{h} \to
\mathfrak{g}$ are called \emph{equivalent} and we denote this by
$r \sim R$ if there exists $\sigma: \mathfrak{h} \to \mathfrak{h}$
a $k$-linear automorphism of $\mathfrak{h}$ such that for any $x$,
$y\in \mathfrak{h}$:
\begin{equation*}\eqlabel{equivLiemaps}
\sigma \bigl([x, \, y]\bigl) - \bigl[\sigma(x), \, \sigma(y)\bigl]
= \sigma(x) \triangleleft R \bigl(\sigma(y)\bigl) + R
\bigl(\sigma(x)\bigl) \rightharpoonup \sigma(y) - \sigma\bigl(x
\triangleleft r(y)\bigl) - \sigma \bigl(r(x) \rightharpoonup
y\bigl)
\end{equation*}
\end{definition}

To conclude this section, the following result provides the answer
to the (CC) problem for Leibniz algebras:

\begin{theorem}\thlabel{clasformelorLie}
Let $\mathfrak{g}$ be a Leibniz subalgebra of $\Xi$,
$\mathfrak{h}$ a $\mathfrak{g}$-complement of $\Xi$ and
$(\mathfrak{g}, \mathfrak{h}, \triangleright, \triangleleft,
\leftharpoonup, \rightharpoonup)$ the associated canonical matched
pair. Then $\sim$ is an equivalence relation on the set $
{\mathcal D}{\mathcal M} \, ( \mathfrak{h}, \mathfrak{g} \, | \, (
\triangleright, \triangleleft, \leftharpoonup, \rightharpoonup )
)$ and the map
$$
{\mathcal H}{\mathcal A}^{2} (\mathfrak{h}, \mathfrak{g} \, | \,
(\triangleright, \triangleleft, \leftharpoonup, \,
\rightharpoonup) ) \, := \, {\mathcal D}{\mathcal M} \,
(\mathfrak{h}, \mathfrak{g} \, | \, (\triangleright,
\triangleleft, \leftharpoonup, \rightharpoonup) )/ \sim \,
\longrightarrow {\mathcal F} (\mathfrak{g}, \, \Xi), \qquad
\overline{r} \mapsto \mathfrak{h}_{r}
$$
is a bijection between ${\mathcal H}{\mathcal A}^{2}
(\mathfrak{h}, \mathfrak{g} \, | \, (\triangleright,
\triangleleft, \leftharpoonup, \rightharpoonup) )$ and the
isomorphism classes of all $\mathfrak{g}$-complements of $\Xi$. In
particular, the factorization index of $\mathfrak{g}$ in $\Xi$ is
computed by the formula:
$$
[\Xi : \mathfrak{g}]^f = | {\mathcal H}{\mathcal A}^{2}
(\mathfrak{h}, \mathfrak{g} \, | \, (\triangleright,
\triangleleft, \leftharpoonup, \rightharpoonup) )|
$$
\end{theorem}

\begin{proof}
Follows from \thref{descrierecomlie} taking into account the fact
that two deformation maps $r$ and $R$ are equivalent in the sense
of \deref{equivLie} if and only if the corresponding Leibniz
algebras $\mathfrak{h}_r$ and $\mathfrak{h}_R$ are isomorphic.
\end{proof}

\begin{example}\exlabel{mpleib}
Let $\mathfrak{h}$ be the abelian Lie algebra of dimension $2$
with basis $\{f_1, f_2\}$ and $\mathfrak{g}$ the Lie algebra with
basis $\{e_1, e_2 \}$ and the bracket: $[e_2, \, e_1] = - [e_1,\,
e_2] = e_2$. Then there exists a matched pair of Leibniz algebras
$(\mathfrak{g}, \, \mathfrak{h}, \, \triangleleft, \,
\triangleright, \, \leftharpoonup,\, \rightharpoonup)$, where the
non-zero values of the actions are given as follows:
\begin{eqnarray*}
f_{1} \triangleleft e_{1} &:=& f_{1}, \quad f_{2} \triangleleft
e_{1} := f_{2}, \quad f_{1} \triangleright e_{1} := e_{2}, \quad
e_{1} \leftharpoonup f_{1} := - e_{2}, \quad e_{1} \rightharpoonup
f_{1} := - f_{1}
\end{eqnarray*}
The bicrossed product $\Xi = \mathfrak{g} \bowtie \mathfrak{h}$
associated to this matched pair is the following $4$-dimensional
Leibniz algebra having $\{e_1, e_2, f_1, f_2\}$ as a basis and the
bracket given by:
$$
[e_2, \, e_1] = - [e_1,\, e_2] = e_2, \, \, [f_{1},\, e_{1}] =
f_{1}+e_{2}, \, \, [e_{1}, f_{1}]= - f_{1} - e_{2}, \, \, [f_{2},
\, e_{1}] = f_{2}
$$
Furthermore, the deformation maps associated with the above
matched pair of Leibniz algebras are given as follows:
\begin{eqnarray*}
\overline{r}_{(\gamma, \, \delta)}&:& \mathfrak{h} \to
\mathfrak{g}, \quad \overline{r}(f_{1}) = \gamma e_{2}, \quad
\overline{r}(f_{2}) = \delta e_{2}\\
r_{(\alpha, \, \beta)}&:& \mathfrak{h} \to \mathfrak{g}, \quad
r(f_{1}) = \alpha e_{2}+ \beta e_{1}, \quad r(f_{2}) = 0
\end{eqnarray*}
for some scalars $\alpha$, $\beta$, $\gamma$, $\delta \in k$. One
can easily see that $\mathfrak{h}_{\overline{r}_{(\gamma, \,
\delta)}}$ coincides with the Lie algebra $\mathfrak{h}$ for all
$\gamma$, $\delta \in k$ while $\mathfrak{h}_{r_{(\alpha, \,
\beta)}}$ has the bracket given by:
$$
[f_{1}, \, f_{2}]_{r_{(\alpha, \, \beta)}} =
[f_{1}, \, f_{1}]_{r_{(\alpha, \, \beta)}} = [f_{2}, f_{2}]_{r_{(\alpha, \, \beta)}} = 0,
\quad [f_{2}, \, f_{1}]_{r_{(\alpha, \, \beta)}} = \beta f_{2}
$$
Therefore, if $\beta = 0$ then $\mathfrak{h}_{r_{(\alpha, \,
\beta)}}$ again coincides with $\mathfrak{h}$. If $\beta \neq 0$,
then for any $\alpha \in k$ and $\beta \in k^*$, the Leibniz
algebra $\mathfrak{h}_{r_{(\alpha, \, \beta)}}$ is isomorphic to
the Leibniz algebra $\mathfrak{k}$ with basis $\{F_{1}, \,
F_{2}\}$ and the bracket given by: $[F_{1},\, F_{1}] = F_{2}$,
$[F_{2},\, F_{1}] = F_{2}$. The isomorphism $\psi: \mathfrak{k}
\to \mathfrak{h}_{r_{(\alpha, \, \beta)}}$ is given by:
$\psi(F_{1}) :=  f_{1} + f_{2}$, $\psi(F_{2}) := \beta f_{2}$.
Since obviously $\mathfrak{k}$ is not isomorphic to the abelian
Lie algebra $\mathfrak{h}$ we obtain that the extension $
\mathfrak{g} \subseteq \Xi$ has factorization index $[\Xi :
\mathfrak{g}]^f = 2$.
\end{example}

\end{document}